\documentclass{article}

\usepackage{amsmath}
\usepackage{amssymb}

\usepackage{color}

\newtheorem{theorem}{Theorem}

\newtheorem{corollary}[theorem]{Corollary}

\newtheorem{lemma}[theorem]{Lemma}

\newtheorem{proposition}[theorem]{Proposition}
\newtheorem{remark}[theorem]{Remark}

\newenvironment{proof}[1][Proof]{\noindent\textbf{#1.} }{\ \rule{0.5em}{0.5em}}

\begin{document}

\title{Asymptotic behaviour of first passage time distributions for L\'{e}vy
processes}
\author{R. A. Doney\thanks{Department of Mathematics, University of Manchester, Manchester M13 9PL, United Kingdom, email: ron.doney@manchester.ac.uk}\ and V. Rivero\thanks{Centro de Investigacion en Matematicas A.C., Calle Jalisco s/n, 36240 Guanajuato, Mexico, email: rivero@cimat.mx}}
\maketitle
\abstract
Let $X$ be a real valued L\'{e}vy process that is in the domain of attraction of a stable law without centering with norming function $c.$ As an analogue of the random walk results in \cite{vw} and \cite{rad} we study the local behaviour of the distribution of the lifetime $\zeta$ under the characteristic measure $\underline{n}$ of excursions away from $0$ of the process $X$ reflected in its past infimum, and of the first passage time of $X$ below $0,$ $T_{0}=\inf \{t>0:X_{t}<0\},$ under $\mathbb{P}_{x}(\cdot ),$ for $x>0,$ in two different regimes for $x,$ viz. $x=o(c(\cdot))$ and $x>D c(\cdot),$ for some $D>0.$ We sharpen our estimates by distinguishing between two types of path behaviour, viz. continuous passage at $T_{0}$ and discontinuous passage. In the way to prove our main results we establish some sharp local estimates for the entrance law of the excursion process associated to $X$ reflected in its past infimum. 

~


\section{Introduction and main results}

Let $X$ be a real valued L\'{e}vy process with law $\mathbb{P}$ and
characteristics $(a,\sigma ,\Pi ).$ We are interested in the \textbf{local}
behaviour of the distribution of the first passage time of $X$ below $0,$
i.e. $T_{0}=\inf \{t>0:X_{t}<0\},$ under $\mathbb{P}_{x}(\cdot ),$ for $x>0.$
We start by investigating the existence of a density for this distribution,
but our main focus is on the asymptotic behaviour of this density, or when
it fails to exist, other local-limit type results, all of which are
analogues of results for random walks in \cite{rad}. We will assume
throughout that under $\mathbb{P}$ neither $X$ nor $-X$ is a subordinator; in the first case the problem has no sense, and in the second case a different approach in needed as our methods rely on the possibility of excursions above the minimum. In addition, since the results for compound Poison processes can be deduced directly from the random walk results in \cite{rad}, we will also assume that $\Pi (\mathbb{R)=}\infty
.$ If additionally $0$ is regular for the half-line $(-\infty ,0)$ under $\mathbb{P%
}$, (we abbreviate this to "$X$ is regular downwards") then $T_{0}\equiv 0$
under $\mathbb{P}_{0},$ so as an analogue of the random walk results in \cite%
{vw} we study the distribution of the lifetime $\zeta =\inf \{t>0:\epsilon
_{t}=0\}$ under the characteristic measure $\underline{n}$ of excursions
away from $0$ of the process reflected in its infimum.

It turns out that we need to distinguish between two types of path
behaviour, viz continuous passage at $T_{0}$ and discontinuous passage. It
is known that the first only has positive probability if $X$ "creeps
downwards" under $\mathbb{P},$ or equivalently the drift $d^{\ast }$ of the
downgoing ladder height process is positive. We start by showing that, on
the event of discontinuous passage, $T_{0}$ admits a density under $\mathbb{P%
}_{x},x>0,$ and a similar result holds in the excursion case. In particular,
the first passage time distribution is absolutely continuous in the case $%
d^{\ast }=0$. However, when $d^{\ast }>0,$ it can happen that on the event
of continuous passage the distribution of $T_{0}$ is singular with respect
to Lebesgue measure. We therefore need to formulate our results differently
in these two situations, and the proofs are also somewhat different.

For the asymptotic results which are the main topic of this paper, we assume
that $X$ is in the domain of attraction of a stable distribution without
centering, that is there exists a deterministic function $c:(0,\infty
)\rightarrow (0,\infty )$ such that%
\begin{equation}
\frac{X_{t}}{c(t)}\xrightarrow[]{\mathcal{D}}Y(1),\quad \text{as}\
t\rightarrow \infty ,  \label{1.0}
\end{equation}%
with $Y(1)$ a strictly stable random variable of parameter $0<\alpha \leq 2,$
and positivity parameter $\rho =\mathbb{P}(Y_{1}>0).$ In this case we will
use the notation $X\in D(\alpha ,\rho ),$ and put $\overline{\rho }=1-\rho .$
Hereafter $(Y_{t},t\geq 0)$ will denote an $\alpha $-stable L\'{e}vy process
with positivity parameter $\rho =\mathbb{P}(Y_{1}>0).$

It is well known that in this case the function $c$ is regularly varying at
infinity with index $1/\alpha $. \textbf{Throughout this paper we will use
the notation }$\eta =1/\alpha .$\textbf{\ }

It is also known that the bivariate downgoing ladder process $(\tau ^{\ast
},H^{\ast })$ is in the domain of attraction of a bivariate $(\overline{\rho 
},\alpha \overline{\rho })$ stable law, and since $\overline{\rho }(t)=%
\mathbb{P}(X_{t}<0)\rightarrow \overline{\rho },$ it follows from Spitzer's
formula that%
\begin{equation}
\underline{n}(\zeta >\cdot )\in RV(-\overline{\rho }),  \label{1.1}
\end{equation}%
where $RV(\beta )$ denotes the class of functions which are regularly
varying with index $\beta $ at $\infty .$ Our first concern is to obtain a
local version of (\ref{1.1}), but we need to consider separately the
contributions coming from continuous and discontinuous passage. So let $\nu
(x)=\mathbb{P}_{x}(C_{0})$ where%
\begin{equation*}
C_{0}=\{X(T_{0}-)=0\}
\end{equation*}%
is the event of crossing level $0$ continuously. Then it is known that many
processes have $\nu (x)\equiv 0,$ e.g. any stable process which is not
spectrally positive, and spectrally positive processes have $\nu (x)\equiv
1, $ provided they do not drift to $\infty .$ The case $0<\nu (x)<1,$ $x>0$
arises if and only if $H^{\ast }$ has a positive drift $d^{\ast }$ and in
this case there is a renewal density $u^{\ast },$ and $\nu (x)=d^{\ast
}u^{\ast }(x),$ so that $\lim_{x\downarrow 0}\nu (x)=1,$ and%
\begin{equation*}
\lim_{x\rightarrow \infty }\nu (x)=\frac{d^{\ast }}{m^{\ast }}:=q\in \lbrack
0,1)
\end{equation*}%
where $m^{\ast }=\mathbb{E}H_{1}^{\ast }=d^{\ast }+\int_{0}^{\infty }%
\overline{\mu }^{\ast }(x)dx.$ (See \cite{jbx}.) However, when $X\in
D(\alpha ,\rho ),$ $\mathbb{E}H_{1}^{\ast }=\infty $ whenever $\alpha 
\overline{\rho }<1;$ when $\alpha \overline{\rho }=1,$ $\mathbb{E}%
H_{1}^{\ast }$ can be finite or infinite, even though the limiting stable
process is spectrally positive. So, except in this special case, $q=0$ and possibly $\mathbb{P}_{_{x}}(C_{0})$
should be negligible.

The following result verifies this intuition, since, except in this special
case, because $c(\cdot )\in RV(\eta )$ and $\underline{n}(\zeta >\cdot )\in
RV(-\overline{\rho }),$ we have $1/c(t)=o(\underline{n}(\zeta >t)).$ In it
we write $p=1-q,$ $f$ for the density of $Y_{1}$ and $\underline{n}%
^{c}(t,\Delta ]$ and $\underline{n}^{d}(t,\Delta ]$ for $\underline{n}(\zeta
\in (t,t+\Delta ],\epsilon (\zeta -)=0)$ and $\underline{n}(\zeta \in
(t,t+\Delta ],\epsilon (\zeta -)>0),$ respectively. The quantity $f(0)$ plays an important role in our estimates, known expressions for it can be found in \cite{zolotarev} equation (2.2.11) or in \cite{sato} equations (14.30-33).

\begin{theorem}
\label{thm1} Suppose $X\in D(\alpha ,\rho )$ and fix any $\Delta _{0}>0:$
then uniformly for $\Delta \in (0,\Delta _{0}]$%
\begin{equation}
\lim_{t\rightarrow \infty }\frac{t\underline{n}(\zeta \in (t,t+\Delta ])}{%
\Delta \underline{n}(\zeta >t)}=\overline{\rho }.  \label{1.2}
\end{equation}%
More precisely, we have,

(i) Whenever $\Pi ((-\infty ,0))>0,$ $\exists h_{0}$ such that $\underline{n}%
^{d}(t,\Delta ]:=\int_{t}^{t+\Delta }h_{0}(s)ds,$ and 
\begin{equation}
\lim_{t\rightarrow \infty }\frac{th_{0}(t)}{\overline{\rho }\underline{n}%
(\zeta >t)}=p.  \label{1.4}
\end{equation}

(ii) When $d^{\ast }>0,$ 
\begin{equation}
\lim_{t\rightarrow \infty }\frac{tc(t)\underline{n}^{c}(t,\Delta ]}{%
\overline{\rho }\Delta }=f(0)d^{\ast },\text{ uniformly for }\Delta \in
(0,\Delta _{0}],  \label{1.3}
\end{equation}%
and in particular, if also $\alpha \overline{\rho }=1,$%
\begin{equation}
\lim_{t\rightarrow \infty }\frac{\underline{n}^{c}(t,\Delta ]}{\overline{%
\rho }\Delta \underline{n}(\zeta >t)}=q.  \label{1.3x}
\end{equation}

(iii) If $\Pi ((-\infty ,0))=0$, then $\underline{n}^{d}(t,\Delta ]\equiv
0, $ $q=1,$ and%
\begin{equation}
\lim_{t\rightarrow \infty }\frac{t\underline{n}^{c}(t,\Delta]}{\Delta 
\overline{\rho }\underline{n}(\zeta >t )}=1,\text{ uniformly for }\Delta
\in (0,\Delta _{0}].  \label{1.5}
\end{equation}
\end{theorem}

For the case $x>0,$ we state here only the analogue of (\ref{1.2}), but in
the sequel we will also state and prove results analogous to (\ref{1.4}), (%
\ref{1.3}), (\ref{1.3x}) and (\ref{1.5}).

We write $U^{\ast }$ for the renewal measure of $H^{\ast }$ and $\tilde{h}%
_{x}(\cdot )$ for the density of the first passage time to $(-\infty ,0)$ of 
$Y$ starting from $x>0.$

\begin{remark}
From now on, the phrase "uniformly in $\Delta "$ will be used as an
abbreviation for "uniformly in $\Delta \in (0,\Delta _{0}]$ for any fixed $%
\Delta _{0}>0"$.
\end{remark}

\begin{theorem}
\label{thm2}Uniformly in $\Delta $ and $x>0$ 
\begin{equation}
\frac{t\mathbb{P}_{x}(T_{0}\in (t,t+\Delta ])}{\Delta }=\tilde{h}%
_{x_{t}}(1)+o(1)\text{ as }t\rightarrow \infty .  \label{1.6}
\end{equation}%
Also, uniformly in $\Delta $ and $x>0$ such that $x_{t}:=x/c(t)\rightarrow
0, $%
\begin{equation}
\lim_{t\rightarrow \infty }\frac{t\mathbb{P}_{x}(T_{0}\in (t,t+\Delta ])}{%
\Delta U^{\ast }(x)\underline{n}(\zeta >t)}=\overline{\rho }.  \label{1.7}
\end{equation}
\end{theorem}

\begin{remark}
If $x_{t}\rightarrow 0$ or $\infty ,$ $\tilde{h}_{x_{t}}(1)\rightarrow 0$
and the RHS of (\ref{1.6}) is $o(1),$ so it is sufficient to show that for
any $D>1$ (\ref{1.6}) holds uniformly in $\Delta $ and $x$ such that $%
x_{t}\in \lbrack D^{-1},D].$
\end{remark}

\begin{remark}
In Lemma \ref{L3} in the next section we will see that if $x_{t}\rightarrow
0 $ then $U^{\ast }(x)\underline{n}(\zeta >t)\rightarrow 0$ and hence the
estimate (\ref{1.7}) is more precise than (\ref{1.6}) in the "small $x_{t}"$
situation.
\end{remark}

We finish this section by stating two propositions which play an important
part in the proof of each of the above results.

Write $U$ for the renewal measure in the upgoing ladder height process $H,$ $%
f$ \ for the density of $Y_{1},$ and $g$  for the probability density
function of the $(\alpha ,\rho )$-stable meander of length $1$ at time $1$.

\begin{proposition}
\label{prop2}Assume that $X\in D(\alpha ,\rho ).$ Then uniformly in $\Delta $
and $y\geq 0$ such that $y_{t}:=y/c(t)\rightarrow 0,$%
\begin{equation}
tc(t)\underline{n}\left( \epsilon _{t}\in (y,y+\Delta ),\zeta >t\right)
\backsim f(0)\int_{y}^{y+\Delta }U(z)dz\text{ as }t\rightarrow \infty .
\label{1.9}
\end{equation}
\end{proposition}

\begin{proposition}
\label{prop1}Assume that $X\in D(\alpha ,\rho ).$ Then uniformly in $\Delta $
and $y\geq 0,$ 
\begin{equation}
c(t)\underline{n}\left( \epsilon _{t}\in (y,y+\Delta )|\zeta >t\right)
=\Delta \left( g\left( y_{t}\right) +o(1)\right) \text{ as }t\rightarrow
\infty .  \label{1.8}
\end{equation}
\end{proposition}

\begin{remark}
If we made the simplifying assumptions that $\alpha \overline{\rho }<1,$
that $X$ is regular upwards and downwards and does not creep downwards, the
following proofs would be considerably simplified. However we believe the
additional work is justified because it would be unnatural to exclude the
case $\alpha \overline{\rho }=1,$ or to make any assumptions about the local
behaviour of $X.$
\end{remark}

\section{Preliminaries}

We recall a few customary notations in fluctuation theory. For background
about fluctuation theory for L\'{e}vy processes the reader is referred to
the books \cite{bertoinbook}, \cite{doneybook}, and \cite{kypbook}.

The process $X_{t}-I_{t}=X_{t}-\inf_{0\leq s\leq t}X_{s},$ $t\geq 0$ is a
strong Markov process, the point process of its excursions out of zero forms
a Poisson point process with intensity or excursion measure $\underline{n}$.
We will denote by $\{\epsilon _{t},t>0\}$ the generic excursion process and
by $\zeta $ its lifetime. It is known that under $\underline{n}$ the
excursion process is Markovian with semigroup given by $\mathbb{P}%
_{x}(X_{t}\in dy,t<T_{0}).$ We will denote by $L^{\ast }$ the local time at $%
0$ for $X-\underline{X},$ and we will assume WLOG that it is normalized so
that 
\begin{equation}
\mathbb{E}\left( \int_{0}^{\infty }e^{-s}dL_{s}^{\ast }\right) =1.
\label{norm}
\end{equation}%
We will denote by $\tau ^{\ast }$ the right continuous inverse of the local
time $L^{\ast },$ and refer to it as the downward ladder time process, and
call $\{H_{t}^{\ast }=-X_{\tau _{t}^{\ast }},t\geq 0\}$ the downward ladder
height process. The potential measure of the bivariate process $(\tau ^{\ast
},H^{\ast })$ will be denoted by 
\begin{equation*}
W^{\ast }(dt,dx)=\int_{0}^{\infty }ds\mathbb{P}(\tau _{s}^{\ast }\in
dt,H_{s}^{\ast }\in dx),\qquad t\geq 0,x\geq 0.
\end{equation*}%
The marginal in space of $W^{\ast }$ is the potential measure of the
downward ladder height process $H^{\ast },$ and we will denote by $U^{\ast }$
its associated renewal function 
\begin{equation*}
U^{\ast }(a):=W^{\ast }([0,\infty )\times \lbrack 0,a])=\int_{0}^{\infty }ds%
\mathbb{P}(H_{s}^{\ast }\leq a),\qquad a\geq 0.
\end{equation*}%
Analogously, the function $V^{\ast }$ will denote the renewal function of
the downward ladder time process, $\tau ^{\ast }.$ We will use a similar
notation for the analogous objects defined in terms of $X^{\ast }$ but we
will remove the symbol\ $\ast $ \ from them, and the excursion measure will
be denoted by $\overline{n}.$

An important duality relation, which we will use extensively, connects $%
W^{\ast }$ and $W$ with the characteristic measures $\underline{n}$ and $%
\overline{n}$: see Lemma 1 in \cite{chaumontsup2010}.

\begin{lemma}
\label{LC}Let $a,a^{\ast }$ denote the drifts in the ladder time processes $%
\tau $ and $\tau ^{\ast }:$ then on $[0,\infty )\times \lbrack 0,\infty )$
we have the identities 
\begin{eqnarray}
W(dt,dx) &=&a^{\ast }\delta _{\{(0,0)\}}(dt,dx)+\underline{n}(\epsilon
_{t}\in dx,\zeta >t)dt,  \label{W*} \\
W^{\ast }(dt,dx) &=&a\delta _{\{(0,0)\}}(dt,dx)+\overline{n}(\epsilon
_{t}\in dx,\zeta >t)dt,  \label{W}
\end{eqnarray}%
so that, in particular%
\begin{eqnarray}
U(x) &=&a^{\ast }+\int_{0}^{\infty }\int_{0}^{x}\underline{n}(\epsilon
_{t}\in dy,\zeta >t)dt,  \label{U} \\
U^{\ast }(x) &=&a+\int_{0}^{\infty }\int_{0}^{x}\overline{n}(\epsilon
_{t}\in dy,\zeta >t)dt.  \label{U*}
\end{eqnarray}
\end{lemma}

\begin{remark}
Note that $a=0$ (respectively $a^{\ast }=0)$ is equivalent to $X$ being regular
downwards (respectively \ upwards$)$, and since we exclude the Compound
Poisson case, at most one of $a$ and $a^{\ast }$ is positive.
\end{remark}

We write $\mathbb{P}^{\ast }$ for the law of the dual L\'{e}vy process $%
X^{\ast }=-X,$ and define

\begin{equation*}
\overline{\Pi}(y)=\Pi(y,\infty ),\quad \overline{\Pi }^{\ast }(y):=\Pi
(-\infty ,-y),\quad y\geq 0.
\end{equation*}%
Put 
\begin{eqnarray*}
h_{x}(t) &=&\mathbb{E}_{x}\left( \overline{\Pi }^{\ast
}(X_{t}),t<T_{0}\right) =\int_{0}^{\infty }\mathbb{P}_{x}(X_{t}\in
dy,t<T_{0})\overline{\Pi }^{\ast }(y),\text{ }x,t>0, \\
h_{0}(t) &=&\underline{n}(\overline{\Pi }^{\ast }(\epsilon _{t}),t<\zeta
)=\int_{0}^{\infty }\underline{n}(\epsilon _{t}\in dy,t<\zeta )\overline{\Pi 
}^{\ast }(y),\text{ }t>0,
\end{eqnarray*}

\begin{lemma}
\label{lem1}We have the following formulae 
\begin{eqnarray}
\mathbb{P}_{x}\left( T_{0}\in dt,X_{T_{0}-}>0\right) &=&h_{x}(t)dt,\text{ }%
x,t>0;  \label{eqf2} \\
\underline{n}(\zeta \in dt,\epsilon _{\zeta -}>0) &=&h_{0}(t)dt,\text{ }t>0.
\end{eqnarray}%
In particular, for $x>0,$ we have that if $X$ does not creep downward then
the law of $T_{0}$ under $\mathbb{P}_{x}$ is absolutely continuous w.r.t
Lebesgue measure.
\end{lemma}

\begin{proof}
Let $f$ \ be a measurable and bounded function. Using the fact that the
jumps of $X$ form a Poisson point process in $\mathbb{R}^{+}\times \mathbb{R}$
with intensity measure $dt\Pi (dz),$ and the compensation formula, we get%
\begin{equation*}
\begin{split}
\mathbb{E}_{x}(f(T_{0})1_{\{X_{T_{0}-}>0\}})& =\mathbb{E}\left(
\sum_{s>0}f(s)1_{\{I_{s-}>-x,\Delta X_{s}<-(x+X_{s-})\}}\right) \\
& =\mathbb{E}\left( \int_{0}^{\infty }dsf(s)1_{\{I_{s-}>-x\}}\Pi \left(
-\infty ,-(x+X_{s-})\right) \right) \\
& =\mathbb{E}\left( \int_{0}^{\infty }dsf(s)1_{\{I_{s}>-x\}}\Pi \left(
-\infty ,-(x+X_{s})\right) \right) \\
& =\mathbb{E}\left( \int_{0}^{\infty }dsf(s)1_{\{s<T_{-x}\}}\Pi \left(
-\infty ,-(x+X_{s})\right) \right) \\
& =\int_{0}^{\infty }dsf(s)\mathbb{E}\left( \Pi \left( -\infty
,-(x+X_{s})\right) 1_{\{s<T_{-x}\}}\right) \\
& =\int_{0}^{\infty }dsf(s)\mathbb{E}_{x}\left( \Pi \left( -\infty
,-X_{s}\right) 1_{\{s<T_{0}\}}\right) .
\end{split}%
\end{equation*}%
{We next prove the identity under $\underline{n}.$ For $t,s\geq 0,$ we have
from the Markov property under $\underline{n}$ that 
\begin{equation*}
\begin{split}
\underline{n}(\zeta >t+s,\epsilon _{\zeta -}>0)& =\underline{n}\left( 
\mathbb{P}_{\epsilon _{s}}(T_{0}>t,X_{T_{0}-}>0),s<\zeta \right) \\
& =\int_{[0,\infty )}\underline{n}(\epsilon _{s}\in dy,s<\zeta
)\int_{t}^{\infty }du\mathbb{E}_{y}(\overline{\Pi }^{-}(X_{u}),u<T_{0}) \\
& =\int_{t}^{\infty }du\int_{[0,\infty )}\underline{n}(\epsilon _{s}\in
dy,s<\zeta )\mathbb{E}_{y}(\overline{\Pi }^{-}(X_{u}),u<T_{0}) \\
& =\int_{t}^{\infty }du\underline{n}\left( \overline{\Pi }^{-}(\epsilon
_{u+s}),u+s<\zeta \right) \\
& =\int_{t+s}^{\infty }dv\underline{n}\left( \overline{\Pi }^{-}(\epsilon
_{v}),v<\zeta \right) .
\end{split}%
\end{equation*}%
}
\end{proof}

In the case where the process creeps downward there is no general result
about the absolute continuity of the law of $T_{0}$ under $\mathbb{P}%
_{x}(\cdot |X_{T_{0}}=0).$ However, if $X$ is spectrally positive then the
downward ladder time process is in fact the first passage time process $%
(T_{-x},x>0)$, where%
\begin{equation*}
T_{-x}=\inf \{t>0:X_{t}<-x\},
\end{equation*}%
and $T_{0}$ under $\mathbb{P}_{x}$ has the same law as $T_{-x}$ under $%
\mathbb{P}.$ By the continuous time version of the "Ballot Theorem" (see
Corollary 3, p190 of \cite{bertoinbook}) it follows that the law of $T_{0}$
is absolutely continuous under $\mathbb{P}_{x}$ for all $x>0$ iff the $%
\mathbb{P}_{x}$ distribution of $X_{t}$ is absolutely continuous for all $%
t>0.$ But Orey \cite{Orey} gave examples where this fails. In these
examples, $X$ has infinite variation, but it is also easy to see that for instance if $%
X_{t}=N_{t}-t,$ $t\geq 0,$ where $N_{t}$ is a Poisson process with parameter 
$\lambda >0,$ the law of $T_{0}$ under $\mathbb{P}_{x}$ is atomic with
support in $\{x+n,n\in \mathbb{Z}^{+}\}.$ (This example is not really
relevant to the sequel, because such a process cannot be in the domain of
attraction of any stable law.)

In order to shorten the notation throughout the rest of the paper we will
understand the following terms as equal, for $s>0,$ 
\begin{equation*}
\underline{n}_{s}(dy)=\underline{n}(\epsilon _{s}\in dy)=\underline{n}%
(\epsilon _{s}\in dy,s<\zeta ),\qquad y>0.
\end{equation*}%
Since in any case we will be integrating over $(0,\infty )$ there will not
be any risk of confusion. Analogous notation will be used under the
excursion measure $\overline{n}$. 

Throughout this paper we will make systematic use of the identities in the
following Lemma \ref{lemma:1} as well as the estimates in the Lemma~\ref{L3}.

\begin{lemma}
\label{lemma:1}
\begin{itemize}
\item[(i)] The semigroup of $X$ can be expressed as: for $x,y\in\mathbb{R}$ 
\begin{equation}\label{lastexit0}
\begin{split}
\mathbb{P}_{x}\left( X_{t}\in dy\right) &=\int_{s=0}^{t}ds\int_{z>(x-y)^{+}}%
\overline{n}_{s}(dz)\underline{n}_{t-s}(dy+z-x)\\ 
 &\ +a\underline{n}_{t}(dy-x)\boldsymbol{1}_{\{y\geq x\}}+a^{\ast }\overline{n%
}_{t}(x-dy)\boldsymbol{1}_{\{y\leq x\}}.
\end{split}
\end{equation}

\item[(ii)] The semigroup of $X$ killed at its first entrance into $(-\infty
,0)$ can be expressed as: for $x,y\in \mathbb{R}^{+}$
\begin{equation}\label{lastexit}
\begin{split}
\mathbb{P}_{x}\left( X_{t}\in dy,t<T_{0}\right)
&=\int_{s=0}^{t}ds\int_{z\in ((x-y)^{+},x]}\overline{n}_{s}(dz)\underline{n}%
_{t-s}\left( dy+z-x\right) \\
&+a\underline{n}_{t}(dy-x)\boldsymbol{1}_{\{y\geq x\}}+a^{\ast }\overline{n%
}_{t}(x-dy)\boldsymbol{1}_{\{y\leq x\}}.
\end{split}
\end{equation}

\item[(iii)] The one-dimensional distribution of the excursion process under 
$\underline{n}$ can be decomposed as: for $x>0$ 
\begin{equation}
t\underline{n}(\epsilon _{t}\in dx)=\int_{0}^{t}ds\int_{z\in \lbrack 0,x]}%
\underline{n}_{s}(dz)\mathbb{P}_{z}(X_{t-s}\in dx)+a^{\ast }\mathbb{P}%
(X_{t}\in dx).  \label{DA}
\end{equation}
\end{itemize}
\end{lemma}

\begin{proof}
The identity (\ref{DA}) is due to Alili and Chaumont \cite{alilichaumont},
and it is a generalization of a result for random walks due to Alili and
Doney \cite{alilidoney}, see also \cite{doneysavov} for further details
about the proof of this result. The proof of the identities (\ref{lastexit0}) and (\ref{lastexit}), together with other useful fluctuation identities
can be found in \cite{chaumontsup2010}. \end{proof}



In what follows, $k,k_{1},k_{2},\cdots $will denote fixed positive constants
whereas $C$ will denote a generic constant whose value can change from line
to line. As previously remarked, the norming function $c(\cdot )\in RV(\eta
),$where $\eta =1/\alpha .$ More precisely we will assume, with no loss of
generality, that $Y$ is a standard stable process, and $c$ can be taken to
be a continuous, monotone increasing inverse of the quantity $x^{2}/m(x);$ where $%
m(x)=\int_{-x}^{x}y^{2}\Pi (dy)$ and necessarily $m(\cdot )\in RV(2-\alpha
) $.  It follows from this that, when $\alpha <2,$ we have $t\overline{\Pi }%
(c(t))\rightarrow k$ and $t\overline{\Pi }^{\ast }(c(t))\rightarrow k^{\ast
},$ with $k^{\ast }>0$ if $\alpha \overline{\rho }<1,$ and $k^{\ast }=0$ if $%
\alpha \overline{\rho }=1,$ when necessarily $k>0.$ Finally when $\alpha =2,$
we have $t(\overline{\Pi }(c(t))+\overline{\Pi }^{\ast }(c(t)))\rightarrow
0, $ so we can take $k=k^{\ast }=0.$

The following local limit theorem is a crucial tool.

\begin{proposition}
\label{local}Assume that $X\in D(\alpha ,\rho ),$ with $\alpha \overline{%
\rho }\leq 1.$ Then uniformly in $\Delta $ and $x\in \mathbb{R},$%
\begin{equation}
c(t)\mathbb{P(}X_{t}\in (x,x+\Delta ])=\Delta (f(\frac{x}{c(t)})+o(1))\text{
as }t\rightarrow \infty .  \label{llt1}
\end{equation}%
Consequently given any $\Delta _{0}>0$ there are constants $k_{0}$ and $%
t_{0} $ such that%
\begin{equation}
c(t)\mathbb{P(}X_{t}\in (x,x+\Delta ])\leq k_{0}\Delta \text{ for all }%
t\geq t_{0}\text{ and }\Delta \in (0,\Delta _{0}].  \label{llt2}
\end{equation}
\end{proposition}

We have not been able to locate this result in the literature, but it can
easily be proved by repeating the argument used for non-lattice random walks
in \cite{stonebkly}, with very minor changes.

Other useful facts are in:

\begin{lemma}
\label{L3}Assume that $X\in D(\alpha ,\rho ),$ with $\alpha \overline{\rho }%
\leq 1.$ We have that 
\begin{equation}
U^{\ast }(c(t))\sim \frac{k_{1}}{\underline{n}(\zeta >t)},\qquad U(c(t))\sim 
\frac{k_{2}}{\overline{n}(\zeta >t)}\qquad t\rightarrow \infty ,  \label{rt}
\end{equation}%
Also 
\begin{equation}
t\underline{n}(\zeta >t)\overline{n}(\zeta >t)\xrightarrow[t\to\infty]{}%
k_{3},  \label{equivtails}
\end{equation}%
where $k_{3}=(\Gamma (\rho )\Gamma (\overline{\rho }))^{-1}.$
\end{lemma}

\begin{proof}
Let $Y^{t}$ the L\'{e}vy process defined by $Y_{s}^{t}:=\frac{X_{ts}}{c(t)}%
,s\geq 0,$ so that $Y^{t}$ converges weakly to $Y$. By a recent result by
Chaumont and Doney~\cite{chaumontanddoney}, see also \cite{vigonthesis}
chapter 3, Lemma 3.4.2, we know that the ladder processes associated to $%
Y^{t} $ also converge weakly towards those associated to $Y.$ Hence the
upward ladder height subordinator associated to $Y^{t}$ converges to a
stable subordinator of parameter $\alpha \rho ,$ where if $\alpha \rho =1$
we interpret the limit as a pure drift. On the one hand, when we write this
in terms of Laplace exponents we get 
\begin{equation*}
\lim_{t\rightarrow \infty }{\kappa }^{(t)}(0,\lambda )=B\lambda ^{\alpha
\rho },\qquad \lambda \geq 0,
\end{equation*}%
where ${\kappa }^{(t)}(\cdot ,\cdot )$ denotes the Laplace exponent of the
upward ladder process associated to $Y^{t},$ and $B$ is a constant depending
on the normalization of the local time, which because of the normalization
chosen here equals $1$. On the other hand, when we write Fristedt's formula
for ${\kappa }^{(t)}$ we get the identities 
\begin{equation*}
\begin{split}
{\kappa }^{(t)}(0,\lambda )& =\exp \left\{ \int_{0}^{\infty }\frac{ds}{s}%
\int_{[0,\infty )}\left( e^{-s}-e^{-\lambda x}\right) \mathbb{P}%
(Y_{s}^{t}\in dx)\right\} \\
& =\exp \left\{ \int_{0}^{\infty }\frac{ds}{s}\int_{[0,\infty )}\left(
e^{-s/t}-e^{-\lambda x/c(t)}\right) \mathbb{P}(X_{s}\in dx)\right\} \\
& =\exp \left\{ \int_{0}^{\infty }\frac{ds}{s}\int_{[0,\infty )}\left(
e^{-s}-e^{-\lambda x/c(t)}\right) \mathbb{P}(X_{s}\in dx)\right\} \\
& \qquad \times \exp \left\{ -\int_{0}^{\infty }\frac{ds}{s}\left(
e^{-s}-e^{-s/t}\right) \mathbb{P}(X_{s}\geq 0)\right\} \\
& =\frac{{\kappa }(0,\lambda /c(t))}{{\kappa }(1/t,0)},
\end{split}%
\end{equation*}%
for $\lambda \geq 0,t>0;$ where $\kappa (\cdot ,\cdot )$ denotes the Laplace
exponent of the upward ladder process $(\tau ,H).$ In particular, since $%
\kappa (0,1)=1,$ 
\begin{equation*}
{\kappa }(0,1/c(t))\sim {\kappa }(1/t,0),\qquad \text{as}\ t\rightarrow
\infty .
\end{equation*}%
By the hypothesis of the Lemma we have ${\kappa }(\cdot ,0)\in RV(\rho )$
and ${\kappa }(0,\cdot )\in RV(\alpha \rho )$. To conclude we use
Proposition~{III.1} in \cite{bertoinbook} to deduce that 
\begin{equation*}
{\kappa }(1/t,0)\sim {\Gamma (1-\rho )}\overline{n}(\zeta >t),\qquad {\kappa 
}(0,1/t)\sim \frac{1}{\Gamma (1+\alpha \rho ){U}(t)},\qquad \text{as}\
t\rightarrow \infty .
\end{equation*}%
It follows that 
\begin{equation*}
U(c(t))\sim \frac{1}{\Gamma (1+\alpha \rho )\Gamma (1-\rho )\overline{n}%
(\zeta >t)},\qquad t\rightarrow \infty .
\end{equation*}%
By applying this result to the dual L\'{e}vy process $-X$ we get the first
asymptotic.

To prove (\ref{equivtails}) we observe that from Lemma \ref{LC} 
\begin{equation*}
V[0,t)=a+\int_{0}^{t}\underline{n}(\zeta >s)ds,\qquad t\geq 0.
\end{equation*}%
Applying again Proposition~{III.1} in \cite{bertoinbook} but this time to the
upward ladder time subordinator we get that 
\begin{equation*}
V[0,t)\sim \frac{1}{\Gamma (1+\rho )\Gamma (1-\rho )\overline{n}(\zeta >t)}%
,\qquad t\rightarrow \infty .
\end{equation*}%
Then by Karamata's theorem we have also that 
\begin{equation*}
\int_{0}^{t}\underline{n}(\zeta >s)ds\sim \frac{1}{1-\overline{\rho }}t%
\underline{n}(\zeta >t),\qquad t\rightarrow \infty .
\end{equation*}%
The result follows by equating the terms.
\end{proof}

A consequence of the fact that $(X(ts)/c(t),s\geq 0)$ converges in law to $%
(Y(s),s\geq 0)$, is that

\begin{lemma}
\label{CD}Assume that $X\in D(\alpha ,\rho ),$ with $\alpha \overline{\rho }%
\leq 1.$ Then as $t\rightarrow \infty $%
\begin{equation*}
\underline{n}(\epsilon _{t}\in c(t)dx|\zeta >t)\overset{D}{\rightarrow }%
\mathbb{P}(Z_{1}\in dx), 
\end{equation*}%
where $Z_{1}$ denotes the stable meander of length 1 at time 1 based on $Y.$
\end{lemma}

\begin{proof}
Let $\mathbb{P}^{\uparrow }$ denote the law of "$X$ conditioned to stay
positive, starting from zero". (For a proper definition of this see e.g.
Chapter 8 of \cite{doneybook}.) Then, using the absolute continuity between $%
\underline{n}$ and $\mathbb{P}^{\uparrow },$ and Lemma~\ref{L3}, we have that over compact sets in $(0,\infty)$
\begin{eqnarray*}
\underline{n}(\epsilon _{t} \in c(t)dx|\zeta >t)&=&\frac{C\mathbb{P}%
^{\uparrow }(X_{t}\in c(t)dx)}{\underline{n}(\zeta >t)U^{\ast }(c(t)x)} \\
&\backsim &\frac{C\mathbb{P}^{\uparrow }(X_{t}\in c(t)dx)}{x^{\alpha 
\overline{\rho }}\underline{n}(\zeta >t)U^{\ast }(c(t))}\backsim \frac{C%
\mathbb{P}^{\uparrow }(X_{t}\in c(t)dx)}{x^{\alpha \overline{\rho }}} \\
&\rightarrow &Cx^{-\alpha \overline{\rho }}\mathbb{P}^{\uparrow }(Y_{1}\in
dx)=C\underline{n}^{Y}(\epsilon (1)\in dx|\zeta >1) \\
&=&C\mathbb{P}(Z_{1}\in dx).
\end{eqnarray*}%
Here the convergence of $\mathbb{P}^{\uparrow }(X_{t}\in
c(t)dx)$ to $\mathbb{P}^{\uparrow }(Y_{1}\in dx)$ is a consequence of
results in \cite{chaumontanddoney}. The above argument is valid over compact sets of $(0,\infty)$, thus proving the vague convergence. To get the convergence in distribution we should also verify that the mass is preserved, but this is straightforward from the fact that $\underline{n}(\epsilon_{t}\in(0,\infty)| \zeta>t)=1=\mathbb{P}(Z_{1}\in(0,\infty))$. This would finish the proof if we can guarantee that $C=1$, but this is a consequence of the normalization chosen.
\end{proof}

\section{Proof of Propositions \protect\ref{prop2} and \protect\ref{prop1}}

We start by proving the following Lemmas.

\begin{lemma}
\label{A}Put $\kappa _{t}^{\Delta }(x)=\underline{n}^{c}(t,\Delta ]+%
\underline{n}^{d}(t,\Delta ]=\underline{n}(\epsilon _{t}\in (x,x+\Delta ])$
and fix $\Delta _{0}>0.$ Then, for all values of $\alpha \overline{\rho }$,
for some constants $k_{4}$ and $t_{0}$ we have, uniformly for $0<\Delta \leq
\Delta _{0}$ and $0\leq x\leq c(t),$%
\begin{equation}
tc(t)\kappa _{t}^{\Delta }(x)\leq k_{4}\Delta U(x+\Delta )\text{ for }t\geq
t_{0}.  \label{3}
\end{equation}
\end{lemma}

\begin{proof}
This is similar to the proof of Lemma 20 in \cite{vw}. First we use the
bound (\ref{llt2}) from Proposition \ref{local} to get 
\begin{eqnarray}
c(t)\kappa _{t}^{\Delta }(x) &=&c(t)\int_{y>0}\underline{n}(\epsilon
(t/2)\in dy)\mathbb{P}_{y}(X_{t/2}\in (x,x+\Delta ],T_{0}>t/2)  \notag \\
&\leq &\frac{k_{0}\Delta c(t)}{c(t/2)}\int_{y>0}\underline{n}(\epsilon
(t/2)\in dy)=\frac{k_{0}\Delta c(t)\underline{n}(\zeta >t/2)}{c(t/2)}  \notag
\\
&\leq &k_{5}\Delta \underline{n}(\zeta >t).  \label{4}
\end{eqnarray}%
Next, it is immediate from equation (\ref{DA}) that%
\begin{equation}
t\kappa _{t}^{\Delta }(x)=\int_{0}^{t}{du}\int_{z=x}^{x+\Delta }\int_{y=0}^{z}%
\mathbb{P(}X_{t-u}\in dy)\underline{n}_{u}(dz-y)+a^{\ast }\mathbb{P}%
(X_{t}\in (x,x+\Delta ]).  \label{ac2}
\end{equation}%
It is useful to note that we can write the inner double integral either as%
\begin{equation*}
\int_{y=0}^{x+\Delta }\mathbb{P(}X_{t-u}\in dy)\underline{n}%
_{u}([(x-y)^{+},x-y+\Delta )),
\end{equation*}%
or as%
\begin{eqnarray*}
&&\int_{w=0}^{x+\Delta }\underline{n}_{u}(dw)\int_{y=(x-w)^{+}}^{x-w+\Delta }%
\mathbb{P(}X_{t-u}\in dy) \\
&=&\int_{w=0}^{x+\Delta }\underline{n}_{u}(dw)\mathbb{P(}X_{t-u}\in \lbrack
(x-w)^{+},x-w+\Delta )).
\end{eqnarray*}%
So we take $\delta \in (0,1)$ and write $t\kappa _{t}^{\Delta
}(x)=J_{1}^{\delta }+J_{2}^{\delta }+a^{\ast }\mathbb{P}(X_{t}\in
(x,x+\Delta ]),$ where 
\begin{eqnarray}
J_{1}^{\delta } &=&\int_{0}^{\delta t}du\int_{w=0}^{x+\Delta }\underline{n}%
_{u}(dw)\mathbb{P(}X_{t-u}\in \lbrack (x-w)^{+},x-w+\Delta )){\color{red},}  \notag \\
J_{2}^{\delta } &=&\int_{\delta t}^{t}du\int_{y=0}^{x+\Delta }\mathbb{P(}%
X_{t-u}\in dy)\underline{n}_{u}([(x-y)^{+},x-y+\Delta )),\text{ and}  \notag
\\
&&a^{\ast }\mathbb{P}(X_{t}\in (x,x+\Delta ])=\frac{a^{\ast }\Delta }{c(t)}%
\{f(x/c(t))+o(1)\}\text{ },  \label{*}
\end{eqnarray}%
where we have used (\ref{llt1}). We see immediately from (\ref{4}) that 
\begin{equation*}
J_{2}^{\delta }\leq \frac{k_{5}\Delta \underline{n}(\zeta >\delta t)}{%
c(\delta t)}\int_{0}^{(1-\delta )t}\mathbb{P(}0<X_{u}\leq x+\Delta )du
\end{equation*}%
From (\ref{lastexit0}) and the subadditivity of $U$ we have, for $y>0$ 
\begin{eqnarray*}
\int_{0}^{(1-\delta )t}\mathbb{P(}0 <X_{u}\leq y)du&=&\int_{0}^{(1-\delta
)t}du\int_{s=0}^{u}ds\int_{z=0}^{\infty }\overline{n}_{s}(dz)\underline{n}%
_{u-s}([(z-y)^{+},z]) \\
&=&\int_{s=0}^{(1-\delta )t}ds\int_{v=0}^{(1-\delta
)t-s}dv\int_{z=0}^{\infty }\overline{n}_{s}(dz)\underline{n}%
_{v}([(z-y)^{+},z]) \\
&\leq &\int_{0}^{(1-\delta )t}ds\int_{z=0}^{\infty }\overline{n}%
_{s}(dz)[U(z)-U((z-y)^{+})] \\
&\leq &U(y)\int_{0}^{(1-\delta )t}ds\overline{n}(\zeta >s)\backsim \frac{%
U(y)(1-\delta )^{\overline{\rho }}t\overline{n}(\zeta >t)}{\overline{\rho }},
\end{eqnarray*}%
and using this with $y=x+\Delta $ and (\ref{equivtails}) gives%
\begin{equation}
\lim \sup_{t\rightarrow \infty }\frac{c(t)J_{2}}{\Delta U(x+\Delta )}\leq 
\frac{k_{5}k_{3}(1-\delta )^{\overline{\rho }}}{\overline{\rho }\delta ^{%
\overline{\rho }+\eta }}.  \label{3.1}
\end{equation}%
For the other term, we again use the bound (\ref{llt2}) to get%
\begin{eqnarray}
J_{1}^{\delta } &\leq &k_{0}\Delta \int_{0}^{t\delta }\frac{du}{c(t-u)}%
\int_{w=0}^{x+\Delta }\underline{n}_{u}(dw)  \notag \\
&\leq &\frac{k_{0}\Delta (U(x+\Delta )-a^{\ast })}{c((1-\delta )t)}\backsim 
\frac{k_{0}\Delta (U(x+\Delta )-a^{\ast })}{(1-\delta )^{\eta }c(t)}.
\label{3.2}
\end{eqnarray}%
Choosing $\delta =1/2,$ the result follows from (\ref{*}), (\ref{3.1}) and (%
\ref{3.2}).
\end{proof}

\begin{corollary}
\label{M}The bound (\ref{3}), with a suitable $k,$ holds uniformly in $x\geq
0.$
\end{corollary}

\begin{proof}
Just note that, by (\ref{4}) $tc(t)\kappa _{t}^{\Delta }(x)\leq k_{5}\Delta t%
\underline{n}(\zeta >t)\backsim k_{3}k_{5}\Delta /\overline{n}(\zeta >t)$
and if $x\geq c(t)$ we have $U(x+\Delta )\geq U(c(t))k_{2}/\overline{n}%
(\zeta >t).$
\end{proof}

We can now prove Proposition \ref{prop2}, which we restate as

\begin{proposition}
\label{prop21}Uniformly in $\Delta $ and uniformly as $x/c(t)\rightarrow 0,$%
\begin{equation*}
tc(t)\kappa _{t}^{\Delta }(x)\backsim f(0)\int_{x}^{x+\Delta
}U(y)dy=:f(0)U^{\Delta }(x)\text{ as }t\rightarrow \infty .
\end{equation*}
\end{proposition}

\begin{proof}
We again use the representation $t\kappa _{t}^{\Delta }(x)=J_{1}^{\delta
}+J_{2}^{\delta }+a^{\ast }\mathbb{P}(X_{t}\in (x,x+\Delta ]),$ but this
time we will be choosing $\delta $ small. Recall that the behaviour of the
final term here is given by (\ref{*}). Using Proposition \ref{local}, we get
that as $t\rightarrow \infty ,$ uniformly in $\Delta $ and $\delta ,$%
\begin{eqnarray*}
J_{1}^{\delta } &\sim &\int_{0}^{\delta t}\frac{f(0)}{c(t-u)}%
du\int_{z=0}^{x+\Delta }\underline{n}_{u}(dz)(x-z+\Delta -(x-z)^{+}) \\
&=&\int_{0}^{\delta t}\frac{f(0)}{c(t-u)}du\int_{z>0}^{x+\Delta }(x-z+\Delta
-(x-z)^{+})W(du,dz).
\end{eqnarray*}%
A simple calculation gives 
\begin{eqnarray*}
&&\int_{z>0}^{x+\Delta }(x-z+\Delta -(x-z)^{+})W(du,dz) \\
&=&\Delta W(du,(0,x])+\int_{y=0}^{\Delta }(\Delta -y)W(du,x+dy) \\
&=&\int_{y=0}^{\Delta }W(du,(0,x+\Delta -y))dy,
\end{eqnarray*}%
and we see that\ $J_{1}^{\delta }$ is asymptotically bounded below by 
\begin{equation*}
\frac{f(0)}{c(t)}\left( U^{\Delta }(x)-\Delta a^{\ast }-\int_{\delta
t}^{\infty }\kappa _{u}^{\Delta }(x)du\right) .
\end{equation*}%
The same argument gives the asymptotic upper bound of 
\begin{equation*}
\frac{f(0)}{c(t(1-\delta ))}\left( U^{\Delta }(x)-\Delta a^{\ast
}-\int_{\delta t}^{\infty }\kappa _{u}^{\Delta }(x)du\right) .
\end{equation*}%
By Corollary \ref{M}, for each fixed $\delta >0$ we have the asymptotic
bound 
\begin{eqnarray*}
\int_{\delta t}^{\infty }\kappa _{u}^{\Delta }(x)du &\leq &k_{4}\Delta
U(x+\Delta )\int_{\delta t}^{\infty }du/(uc(u)) \\
&\backsim &C(\delta )\Delta U(x+\Delta )/c(t) \\
&=&o(1)\Delta U(x+\Delta )=o(1)U^{\Delta }(x),
\end{eqnarray*}%
where we observe that Erickson's \cite{erickson} bounds give
\begin{equation*}
\frac{\Delta U(x+\Delta )}{U^{\Delta }(x)}\leq \left( 
\begin{array}{cc}
\frac{\Delta U(x+\Delta )}{\Delta U(x)}\leq C & \text{for }\Delta \leq x, \\ 
\frac{\Delta U(2\Delta )}{\Delta /2U(\Delta /2)}\leq C & \text{for }x\leq
\Delta .%
\end{array}%
\right.
\end{equation*}%
Hence 
\begin{equation*}
c(t)J_{1}^{\delta }+a^{\ast }\mathbb{P}(X_{t}\in (x,x+\Delta ])\overset{%
t,\delta }{\backsim }f(0)U^{\Delta }(x),
\end{equation*}%
where the notation $A\overset{t,\delta }{\backsim }B$ is shorthand for%
\begin{equation*}
\lim_{\delta \downarrow 0}\limsup_{t\rightarrow \infty }\frac{A}{B}%
=\lim_{\delta \downarrow 0}\liminf_{t\rightarrow \infty }\frac{A}{B}=1.
\end{equation*}%
Also, for each fixed $\delta >0,$ 
\begin{eqnarray*}
J_{2}^{\delta } &\leq &\int_{0}^{(1-\delta )t}\int_{y=0}^{x+\Delta }\mathbb{%
P(}X_{u}\in dy)\kappa _{t-u}^{\Delta }((x-y)^{+})du \\
&\leq &\frac{C(\delta )\Delta }{tc(t)}\int_{0}^{(1-\delta
)t}\int_{y=0}^{x+\Delta }\mathbb{P(}X_{u}\in dy)U((x-y)^{+}+\Delta )du \\
&\leq &\frac{C(\delta )\Delta U(x+\Delta )}{tc(t)}\int_{0}^{(1-\delta )t}%
\mathbb{P(}X_{u}\in (0,x+\Delta ])du.
\end{eqnarray*}%
Since 
\begin{eqnarray*}
\int_{0}^{(1-\delta )t}\mathbb{P(}X_{u} \in (0,x+\Delta ])du&\leq&
t_{0}+\int_{t_{0}}^{(1-\delta )t}\mathbb{P(}X_{u}\in (0,x+\Delta ])du \\
&\leq &t_{0}+C\int_{t_{0}}^{(1-\delta )t}\frac{x+\Delta }{c(u)}du\\ &\leq& t_{0}+%
\frac{C(\delta )(x+\Delta )t}{c(t)}=o(t),
\end{eqnarray*}%
the result follows.
\end{proof}

\begin{corollary}
\label{Q}Uniformly for $0\leq x\leq y=o(c(t))$ we have 
\begin{equation*}
tc(t)\underline{n}_{t}((x,y])\backsim f(0)\int_{x}^{y}U(y)dy\text{ as }%
t\rightarrow \infty .
\end{equation*}
\end{corollary}

\begin{proof}
If $y\leq x=1$ this is immediate from Proposition \ref{prop21}, and
otherwise we split $(x,y]$ into disjoint intervals of length $\leq 1$ and
apply the same proposition to each interval$.$
\end{proof}

In preparation for the next proof, we have:

\begin{lemma}
\label{mea}The density function ${g}$ of the stable meander $Z_{1}$ satisfies
the identity%
\begin{equation}
{g}(x)=\int_{0}^{1}ds\int_{y=0}^{x}s^{-\eta -\overline{\rho }}{g}(s^{-\eta
}y)f_{1-s}((x-y))dy,  \label{id}
\end{equation}%
where $f_{t}$ denotes the density function of $Y_{t}.$
\end{lemma}

\begin{proof}
We recall from \cite{bertoinbook} VIII.4 that the one dimensional law of the
stable meander of length one can be written in terms of the excursion
measure, $\underline{n}^{Y},$ of the stable process $Y$ reflected in its
past infimum by the formula 
\begin{equation}
{g}(z)dz=\mathbb{P}(Z_{1}\in dz)=\underline{n}^{Y}(\epsilon _{1}\in dz|\zeta
>1),\qquad z\geq 0.  \label{stablemeander}
\end{equation}%
But the measure $\underline{n}^{Y}$ inherits the scaling property of the
stable process in the form: for any $c>0,$ and $s>0,$ 
\begin{equation}
\underline{n}^{Y}(\epsilon _{s}\in dy,s<\zeta )=c^{-\overline{\rho }}%
\underline{n}^{Y}(\epsilon _{s/c}\in c^{-\eta }dy,s<c\zeta ),\qquad y>0,
\label{ss}
\end{equation}%
see \cite{bertoinbook} Lemma VIII.14 or \cite{rivero2005a} for a proof of
this fact. Thus 
\begin{eqnarray}
s^{-\eta -\overline{\rho }}{g}(s^{-\eta }y)dy &=&s^{-\overline{\rho }}%
\underline{n}^{Y}(\epsilon _{1}\in s^{-\eta }dy|\zeta >1)  \notag \\
&=&s^{-\overline{\rho }}\underline{n}^{Y}(\epsilon _{s}\in dz|\zeta >s)=%
\frac{\underline{n}^{Y}(\epsilon _{s}\in dy)}{\underline{n}^{Y}(\zeta >1)},
\label{23}
\end{eqnarray}%
and multiplying (\ref{id}) by $\underline{n}^{Y}(\zeta >1)dx$ we see that it
reads%
\begin{equation*}
\underline{n}^{Y}(\epsilon _{1}\in dx)=\int_{0}^{1}ds\int_{y=0}^{x}%
\underline{n}^{Y}(\epsilon _{s}\in dy)f_{1-s}((x-y))dx,
\end{equation*}%
and this is equation (\ref{DA}) specialised to the stable case and $t=1.$
\end{proof}

We can now prove Proposition \ref{prop1}, which we restate;

\begin{proposition}
\label{F}For all values of $\alpha \overline{\rho }$, uniformly for $%
x_{t}\geq 0$ and uniformly in $\Delta ,$%
\begin{equation*}
\frac{c(t)\kappa _{t}^{\Delta }(x)}{\underline{n}(\zeta >t)}=\Delta
({g}(x_{t})+o(1))\text{ as }t\rightarrow \infty .
\end{equation*}
\end{proposition}

\begin{proof}
This time we write
\begin{eqnarray*}
t\kappa _{t}^{\Delta }(x) &=&\int_{0}^{t}\int_{y=0}^{x+\Delta }\int_{z=x\vee
y}^{x+\Delta }\mathbb{P(}X_{u}\in dy)\underline{n}_{t-u}(dz-y)du+a^{\ast }%
\mathbb{P}(X_{t}\in (x,x+\Delta ]) \\
&=&\int_{0}^{t}\int_{y=0}^{x}\mathbb{P(}X_{u}\in dy)\kappa _{t-u}^{\Delta
}(x-y)du+a^{\ast }\mathbb{P}(X_{t}\in (x,x+\Delta ]) \\
 &&+\int_{0}^{t}\int_{y=x}^{x+\Delta }\mathbb{P(}X_{u}\in dy)\underline{n}%
_{t-u}((0,x-y+\Delta ])du \\
&:=&I_{1}+a^{\ast }\mathbb{P}(X_{t}\in (x,x+\Delta ])+I_{2}.
\end{eqnarray*}%
It is immediate from (\ref{llt2}) that $t^{-1}\Delta ^{-1}\mathbb{P}%
(X_{t}\in (x,x+\Delta ])=o(\underline{n}(\zeta >t)/c(t)).$ Also 
\begin{eqnarray*}
I_{2} &=&\int_{0}^{t}\int_{0}^{\Delta }\mathbb{P(}X_{u}\in x+dz)\underline{n}%
_{t-u}((0,\Delta -z])du \\
&\leq &\int_{0}^{t}\mathbb{P(}X_{u}\in (x,x+\Delta ])\kappa _{t-u}^{\Delta
}(0)du \\
&=&\int_{0}^{(1-\delta )t}+\int_{(1-\delta )t}^{t}\mathbb{P(}X_{u}\in
(x,x+\Delta ])\kappa _{t-u}^{\Delta }(0)du \\
&\leq &\frac{k_{4}(1-\delta )t\Delta U(\Delta )}{\delta tc(\delta t)}+\frac{%
k_{0}\Delta }{c((1-\delta )t)}\int_{0}^{\delta t}\underline{n}(\zeta >u)du \\
&\backsim &\frac{k_{4}(1-\delta )\Delta U(\Delta _{0})}{\delta ^{1+^{\eta
}}c(t)}+\frac{k_{0}\Delta \delta ^{1-\overline{\rho }}t\underline{n}(\zeta
>t)}{(1-\delta )^{\eta }c(t)},
\end{eqnarray*}%
so we see that $\lim_{t\rightarrow \infty }\frac{c(t)I_{2}}{t\Delta 
\underline{n}(\zeta >t)}\leq k_{0}\delta ^{\rho }(1-\delta )^{-\eta },$
uniformly in $x,$ and since $\delta $ is arbitrary, $\lim_{t\rightarrow
\infty }\frac{c(t)I_{2}}{t\Delta \underline{n}(\zeta >t)}=0.$ Next put $%
I_{1}=I_{1}^{1}+I_{1}^{2}+I_{1}^{3}$, where by the bound (\ref{4}), for
large enough $t$
\begin{eqnarray*}
I_{1}^{1} &:=&\int_{0}^{\delta t}\int_{y=0}^{x}\mathbb{P(}X_{u}\in dy)\kappa
_{t-u}^{\Delta }(x-y)du \\
&\leq &\frac{k_{5}\Delta \underline{n}(\zeta >(1-\delta )t)}{c((1-\delta )t)}%
\int_{0}^{\delta t}\mathbb{P(}0<X_{u}\leq x)du \\
&\leq &\frac{k_{5}\Delta \delta t\underline{n}(\zeta >(1-\delta )t)}{%
c((1-\delta )t)}\backsim \frac{k_{5}\delta \Delta \underline{n}(\zeta >t)t}{%
(1-\delta )^{\overline{\rho }+\eta }c(t)}.
\end{eqnarray*}%
Also
\begin{eqnarray*}
I_{1}^{3} &:=&\int_{0}^{\delta t}\int_{z{>}0}^{x}\underline{n}_{u}(dz)\mathbb{%
P(}X_{t-u}\in ((x-z)^{+},x-z+\Delta ])du \\
&\leq &\frac{k_{0}\Delta }{c((1-\delta )t)}\int_{0}^{\delta t}\underline{n}%
(\zeta >u)du\backsim \frac{k_{0}\Delta \delta ^{\rho }t\underline{n}(\zeta
>t)}{\rho (1-\delta )^{\eta }c(t)}.
\end{eqnarray*}%
So $\lim_{\delta \rightarrow 0}\lim \sup_{t\rightarrow \infty }\frac{%
c(t)(I_{1}^{1}+I_{1}^{3})}{t\Delta \underline{n}(\zeta >t)}=0.$ For the
other term, using Proposition \ref{local}, we have%
\begin{eqnarray*}
I_{1}^{2} &=&\int_{\delta t}^{(1-\delta )t}\int_{z{>}0}^{x}\underline{n}%
_{u}(dz)\mathbb{P(}X_{t-u}\in ((x-z)^{+},x-z+\Delta ])du \\
&=&\Delta \left( \int_{\delta t}^{(1-\delta )t}\int_{z{>}0}^{x}\underline{n}%
_{u}(dz)f((x-z)/c(t-u))/c(t-u)\right) \\
&&+o\left( \Delta \int_{\delta t}^{(1-\delta )t}\int_{z{>}0}^{x}\frac{%
\underline{n}_{u}(dz)}{c(t-u)}du\right) .
\end{eqnarray*}%
It is easily seen that, for any fixed $\delta >0,$ the error term is $%
o(t\Delta \underline{n}(\zeta >t)/c(t)),$ and in the remaining term we write 
$x=c(t)x_{t},$ $z=c(t)y$ and $u=st$ to see that $\frac{c(t)I_{2}}{\Delta t%
\underline{n}(\zeta >t)}$ can be written as 
\begin{eqnarray*}
&&\frac{1}{t\underline{n}(\zeta >t)}\int_{\delta t}^{(1-\delta )t}\frac{%
c(t)du}{c(t-u)}\int_{y=0}^{x_{t}}\underline{n}(\epsilon _{u}\in c(t)dy)f(%
\frac{c(t)(x_{t}-y)}{c(t-u)})+o(1) \\
&=&\frac{1}{\underline{n}(\zeta >t)}\int_{\delta }^{(1-\delta )}\frac{c(t)ds%
}{c(t(1-s))}\int_{y=0}^{x_{t}}\underline{n}(\epsilon _{ts}\in c(t)dy)f(\frac{%
c(t)(x_{t}-y)}{c(t(1-s)})+o(1) \\
&=&\int_{\delta }^{(1-\delta )}\frac{\underline{n}(\zeta >ts)c(t)ds}{%
\underline{n}(\zeta >t)c(t(1-s))}\int_{y=0}^{x_{t}}\underline{n}(\epsilon
_{ts}\in c(t)dy|\underline{n}(\zeta >ts))f(\frac{c(t)(x_{t}-y)}{c(t(1-s)}%
)+o(1).
\end{eqnarray*}%
It then follows from Lemma \ref{CD}, the regular variation of $\underline{n}%
(\zeta >t)$ and the fact that $f$ is uniformly continuous that this last
expression can be written as%
\begin{eqnarray*}
&&\int_{\delta }^{(1-\delta )}\frac{ds}{s^{\overline{\rho }}(1-s)^{\eta }}%
\int_{y=0}^{x_{t}}{\mathbb{P}}(Z_{1}\in s^{-\eta }dy)f((x_{t}-y)(1-s)^{-\eta })+o(1) \\
&=&\int_{\delta }^{(1-\delta )}ds\int_{y=0}^{x_{t}}s^{-\overline{\rho }%
}{\mathbb{P}}(Z_{1}\in s^{-\eta }dy)f_{1-s}((x_{t}-y))+o(1),
\end{eqnarray*}%
where we have used the scaling property. It is easy to check, from the known
behaviour of $f$ and that of the density of $Z_{s},$ see \cite{doneysavov},
that the corresponding integrals over $(0,\delta )$ and $(1-\delta ,1)$ are $%
o(1)$ as $\delta \rightarrow 0$ uniformly for $x\geq 0,$ so the result
follows from Lemma \ref{mea}.
\end{proof}

\begin{corollary}
\label{xx}Uniformly for $x,y\geq 0,$%
\begin{equation*}
\lim \sup_{t\rightarrow \infty }c(t)\underline{n}(\epsilon _{t}\in
(x,x+y]|\zeta >t)\leq \overline{g}y,
\end{equation*}%
where $\overline{g}=\sup_{x\geq 0}{g}(x)<\infty .$
\end{corollary}

\begin{proof}
If $y\leq 1$ this is immediate from Proposition \ref{F}, and otherwise we
get the conclusion by writing $(x,x+y]$ as the union of disjoint intervals
of length less than or equal to $1.$
\end{proof}

\section{Proof of Theorem \protect\ref{thm1}}

\subsection{The discontinuous case.}

Here we assume $\Pi (\mathbb{R}^{-})>0,$ and deal separately with the cases $%
\alpha \overline{\rho }<1$ and $\alpha \overline{\rho }=1.$

\subsubsection{The case $\protect\alpha \overline{\protect\rho }<1$}

Write, for $y\geq 0,$%
\begin{eqnarray*}
\theta (t,y) &=&\underline{n}(\overline{\Pi }^{\ast }(y+\epsilon_{t}),\zeta
>t)\text{ and} \\
\chi (t,y) &=&\overline{n}(\overline{\Pi }^{\ast }(y-\epsilon _{t})%
\boldsymbol{1}_{\{y>\epsilon _{t}\}},\zeta >t)
\end{eqnarray*}%
so that $\theta (t,0)=h_{0}(t)$ is the density of $\underline{n}(\zeta \in
dt,\epsilon (\zeta-)>0).$ Note that, from e.g. the quintuple identity of \cite%
{dk} {or integrating (\ref{lastexit})}, we have that, for $x>0,$%
\begin{equation}
h_{x}(t)=\int_{0}^{t}\int_{0}^{x}\overline{n}_{s}(x-dy)\theta
(t-s,y)ds+a\theta (t,x)+a^{\ast }\chi (t,y).  \label{x}
\end{equation}%
So as well as proving the result (\ref{1.4}) for $h_{0},$ the following
Proposition will be useful for the case $x>0.$

\begin{proposition}
\label{Z}Assume $\alpha \overline{\rho }<1.$ Then uniformly for $y\geq 0,$%
\begin{equation}
\theta (t,y)\thicksim \overline{\rho }t^{-1}\underline{n}(\zeta >t)\phi
(y_{t})\text{ as }t\rightarrow \infty ,  \label{1}
\end{equation}%
where $y_{t}:=y/c(t)$ and%
\begin{equation}\label{2}
\begin{split}
\phi (z)&=\int_{0}^{\infty }(z+w)^{-\alpha }{g}(w)dw/\int_{0}^{\infty
}w^{-\alpha }{g}(w)dw\\
&=\mathbb{E}\{(z+Z_{1})^{-\alpha }\}/\mathbb{E}%
(Z_{1}^{-\alpha }).  
\end{split}
\end{equation}
\end{proposition}
Our argument to prove Proposition \ref{Z} relies on the decomposition, for $\delta>0$ 
\begin{eqnarray*}
\theta (t,y) &=&\int_{x>0}\underline{n}_{t}(dx)\overline{\Pi }^{\ast }(x+y)
\\
&=&\int_{\delta c(t)\geq x>0}\underline{n}_{t}(dx)\overline{\Pi }^{\ast
}(x+y)+\int_{x>\delta c(t)}\underline{n}_{t}(dx)\overline{\Pi }^{\ast }(x+y)
\\
&:&=I_{1}(\delta ,y)+I_{2}(\delta ,y).
\end{eqnarray*}%
To deal with the first of these we need the following result.

\begin{lemma}
\label{B}For any L\'{e}vy process, $xU(x)\Pi ^{\ast }(dx)$ is integrable at
zero.
\end{lemma}

\begin{proof}
By Vigon's identity, the tail of the L\'{e}vy measure of the down going
ladder height process is given by 
\begin{eqnarray}
\overline{\mu }^{\ast }(x) &=&\int_{0}^{\infty }U(dy)\overline{\Pi }^{\ast
}(x+y)  \label{v} \\
&=&\int_{0}^{\infty }U(dy)\int_{x+y}^{\infty }\Pi ^{\ast }(dz)  \notag \\
&=&\int_{z=x}^{\infty }\Pi ^{\ast }(dz)\int_{y<x-z}^{\infty
}U(dy)=\int_{z=x}^{\infty }\Pi ^{\ast }(dz)U(z-x)  \notag
\end{eqnarray}%
So%
\begin{eqnarray*}
C &=&\int_{0}^{1}\overline{\mu }^{\ast }(x)dx\geq
\int_{x=0}^{1}\int_{z=x}^{1}\Pi ^{\ast }(dz)U(z-x)dx \\
&=&\int_{z=0}^{1}\int_{x=0}^{z}\Pi ^{\ast }(dz)U(z-x)dx \\
&=&\int_{z=0}^{1}\Pi ^{\ast }(dz)\int_{y=0}^{z}U(y)dy \\
&\geq &\int_{z=0}^{1}\Pi ^{\ast }(dz)\int_{y=z/2}^{z}U(y)dy \\
&\geq &\frac{1}{2}\int_{z=0}^{1}zU(z/2)\Pi ^{\ast }(dz).
\end{eqnarray*}%
But by Erickson's \cite{erickson} bounds, $U(z/2)\geq CU(z),$ and the result follows.
\end{proof}

Now we show that uniformly in $y\geq 0$%
\begin{equation*}
\lim_{\delta \downarrow 0}\lim \sup_{t\rightarrow \infty }\frac{%
tI_{1}(\delta ,y)}{\underline{n}(\zeta >t)}=0.
\end{equation*}

First we note that for all $y\geq 0$, we have $I_{1}(\delta ,y)\leq
I_{1}(\delta ,0).$ Then from Lemma \ref{Q}, we can choose $\delta $ small
enough and $t_{0}$ large enough such that $$tc(t)\underline{n}_{t}((x,\delta
c(t))\leq 2f(0)\int^{\delta c(t)}_{x}U(y)dy,\qquad \text{for all}\ 0\leq x\leq \delta c\left( t\right).$$ And then 
\begin{eqnarray*}
&&\int_{0}^{\delta c(t)}\overline{\Pi }^{\ast }(x)\underline{n}_{t}(dx) \\
&=&\overline{\Pi }^{\ast }(\delta c(t))\underline{n}_{t}((0,\delta
c(t))+\int_{0}^{\delta c(t)}\underline{n}_{t}((x,\delta c(t))\Pi ^{\ast }(dx)
\\
&\leq &\frac{2f(0)}{tc(t)}\left( \overline{\Pi }^{\ast }(\delta
c(t))\int_{0}^{\delta c(t)}U(y)dy+\int_{0}^{\delta c(t)}\Pi ^{\ast
}(dx)\int_{x}^{\delta c(t)}U(y)dy\right) \\
&=&\frac{2f(0)}{tc(t)}\int_{0}^{\delta c(t)}\overline{\Pi }^{\ast
}(x)U(x)dx\backsim \frac{C\delta c(t)\overline{\Pi }^{\ast }(\delta
c(t))U(\delta c(t))}{tc(t)}{\color{red},}
\end{eqnarray*}%
where we use the fact that $\overline{\Pi }^{\ast }(t)U(t)$ is rv with index 
$-\alpha +\alpha \rho =-\alpha \overline{\rho }>-1.$ For the same reason, and using Lemma \ref{L3} we
can replace the numerator by 
\begin{eqnarray*}
C\delta ^{1-\alpha +\alpha \rho }c(t)\overline{\Pi }^{\ast }(c(t))U(c(t))
&\backsim &C\delta ^{1-\alpha +\alpha \rho }c(t)t^{-1}t\underline{n}(\zeta
>t) \\
&=&C\delta ^{1-\alpha +\alpha \rho }c(t)\underline{n}(\zeta >t),
\end{eqnarray*}%
and the conclusion follows.

Next we show that for any fixed $b\geq 0$%
\begin{equation}
\lim_{\delta \downarrow 0}\lim_{t\rightarrow \infty }\frac{tI_{2}(\delta
,bc(t))}{\underline{n}(\zeta >t)}=\overline{\rho }\phi (b).  \label{21}
\end{equation}%
For this, we use Lemma \ref{CD} and write%
\begin{eqnarray*}
\frac{tI_{2}(\delta ,bc(t))}{\underline{n}(\zeta >t)} &=&t\int_{x>\delta
c(t)}\underline{n}(\epsilon_{t}\in dx|\zeta >t)\overline{\Pi }^{\ast
}(x+bc(t)) \\
&=&t\int_{y>\delta }\underline{n}(\epsilon_{t}\in c(t)dy|\zeta >t)%
\overline{\Pi }^{\ast }(c(t)(y+b)) \\
&\rightarrow &k^{\ast }\int_{y>\delta }\mathbb{P}(Z_{1}\in dy)(y+b)^{-\alpha
}dy.
\end{eqnarray*}%
By letting $\delta \rightarrow 0$ we see that (\ref{21}) holds, except that $%
\overline{\rho }$ is replaced by $k^{\ast }\mathbb{E}Z_{1}^{-\alpha }.$
Taking $b=0$ this shows that $h_{0}(t)\backsim k^{\ast }\mathbb{E}%
Z_{1}^{-\alpha }t^{-1}\underline{n}(\zeta >t),$ and, as we show later, $%
\underline{n}^{d}(\zeta >t)=\int_{t}^{\infty }h_{0}(s)ds\backsim \underline{n%
}(\zeta >t).$ By applying this result to the case where $X$ is an $\alpha$-stable process with positivity parameter $\rho$ we get that 
\begin{equation}
\overline{\rho }=k^{\ast }\mathbb{E}Z_{1}^{-\alpha }.  \label{22}
\end{equation}%
We have shown that (\ref{1}) holds for $y=bc(t)$. The general result
then follows from the fact that $\theta (t,y)$ is monotone in $y.$

\section{The case $\protect\alpha \overline{\protect\rho }=1.$}

In this case the ladder height process $H^{\ast }$ is relatively stable,
i.e. there is a norming function $b$ such that $H_{t}^{\ast }/b(t)\overset{P}%
{\rightarrow }1,$ and this can happen in two different ways. Put $A^{\ast
}(x)=\int_{0}^{x}\overline{\mu }^{\ast }(y)dy;$ then either $\mathbb{E}%
H_{1}^{\ast }=d^{\ast }+A^{\ast }(\infty )<\infty ,$ or $A^{\ast }(\infty
)=\infty ,$ and in the latter case $A^{\ast }\in RV(0).$ It is immediate
from Vigon's identity that if we put $B(x)=\int_{0}^{x}U(y)\overline{\Pi }%
^{\ast }(y)dy,$ then $A^{\ast }(\infty )<\infty $ iff $B(\infty )<\infty .$ In
our case the connection between these functions is closer than this, because:

\begin{lemma}
\label{L}If $\alpha \overline{\rho }=1$ and $\mathbb{E}H_{1}^{\ast }=\infty $
then $B(x)\backsim A^{\ast }(x)$ as $x\rightarrow \infty .$
\end{lemma}

\begin{proof}
Integrating Vigon's identity gives%
\begin{eqnarray*}
A^{\ast }(x) &=&\int_{0}^{x}\int_{0}^{\infty }U(dz)\overline{\Pi }^{\ast
}(y+z)dy \\
&=&\int_{0}^{\infty }U(dz)\int_{0}^{x}\overline{\Pi }^{\ast
}(y+z)dy=\int_{0}^{\infty }U(dz)\int_{z}^{x+z}\overline{\Pi }^{\ast }(w)dw \\
&=&\int_{0}^{\infty }\overline{\Pi }^{\ast
}(w)dw\int_{(w-x)^{+}}^{w}U(dz)=B(x)+E(x),
\end{eqnarray*}%
where $E(x)=\int_{x}^{\infty }\overline{\Pi }^{\ast
}(w)dw\int_{w-x}^{w}U(dz).$ If we put $\overline{U}(x)=\int_{0}^{x}U(y)dy$
an integration by parts gives%
\begin{eqnarray*}
E(x) &=&\int_{x}^{\infty }\Pi ^{\ast }(dy)\int_{x}^{y}\{U(w)-U(w-x)\}dw \\
&=&\int_{x}^{\infty }\Pi ^{\ast }(dy)\{\overline{U}(y)-\overline{U}(x)-%
\overline{U}(y-x)\} \\
&\leq &x\int_{x}^{\infty }\Pi ^{\ast }(dy)U(y)=x\{\overline{\Pi }^{\ast
}(x)U(x)+\int_{x}^{\infty }\overline{\Pi }^{\ast }(y)U(dy)\}.
\end{eqnarray*}%
Since $A^{\ast \prime }(x)=\overline{\mu }^{\ast }(x)$ and $A^{\ast }\in
RV(0)$ we know that $x\overline{\mu }^{\ast }(x)=o(A^{\ast }(x))$ as $%
x\rightarrow \infty .$ Also%
\begin{eqnarray*}
\overline{\mu }^{\ast }(x) &=&\int_{0}^{\infty }U(dz)\overline{\Pi }^{\ast
}(x+z)\geq \int_{0}^{x}U(dz)\overline{\Pi }^{\ast }(x+z) \\
&\geq &U(x)\overline{\Pi }^{\ast }(2x)\geq CU(2x)\overline{\Pi }^{\ast }(2x),
\end{eqnarray*}%
where we have used Erickson's \cite{erickson} bounds for $U.$ Thus $x\overline{\Pi }^{\ast
}(x)U(x)\leq Cx\overline{\mu }^{\ast }(x/2)=o(A^{\ast }(x)).$ Hence%
\begin{equation*}
x\int_{x}^{\infty }\overline{\Pi }^{\ast }(y)U(dy)=o\left( x\int_{x}^{\infty
}\frac{A^{\ast }(y)U(dy)}{yU(y)}\right) ,
\end{equation*}%
and we can bound the bracketed term on the RHS by%
\begin{equation*}
x\sup_{y\geq x}\left( \frac{A^{\ast }(y)y^{\beta }}{U(y)}\right) \bullet
\int_{x}^{\infty }\frac{U(dy)}{y^{1+\beta }},
\end{equation*}%
where we choose $\beta =\alpha \rho /2$ and recall that $U\in RV(\alpha \rho
).$ From standard properties of regularly varying functions we see that this
last expression is asymptotically equivalent to%
\begin{equation*}
Cx\frac{A^{\ast }(x)x^{\beta }}{U(x)}\bullet \frac{U(x)}{x^{1+\beta }}%
=CA^{\ast }(x),
\end{equation*}%
so we can conclude that $E(x)/A^{\ast }(x)\rightarrow 0,$ which gives the
result.
\end{proof}

This result immediately implies that the function $B(c(t))$ is monotone and
slowly varying. It is therefore possible to find $\delta _{t}\downarrow 0$
such that $\delta _{t}c(t)\rightarrow \infty $ and $L(t):=B(\delta
_{t}c(t))\backsim B(c(t))$ is also slowly varying. Moreover, since for each
fixed $\delta $ we have $t\overline{\Pi }^{\ast }(\delta c(t))=o(t\overline{%
\Pi }(\delta c(t))=o(1),$ we can also arrange that $t\overline{\Pi }^{\ast
}(\delta _{t}c(t))\rightarrow 0.$

\begin{proposition}
\label{K}Define, for $y\geq 0,$ the function 
\begin{equation*}
\psi (y,t)=\int_{0}^{\delta _{t}c(t)}U(z)\overline{\Pi }^{\ast }(z+y)dz,
\end{equation*}%
and note that $\psi (0,t)=L(t).$ Then we have the estimate, uniform for $%
y\geq 0,$ 
\begin{equation*}
\theta (t,y)=\frac{\overline{\rho }\psi (y,t)\underline{n}^{d}(\zeta >t)}{%
tL(t)}+o(t^{-1}\underline{n}(\zeta >t))\text{ as }t\rightarrow \infty .
\end{equation*}%
In particular, $h_{0}(t)\backsim \overline{\rho }t^{-1}\underline{n}%
^{d}(\zeta >t).$
\end{proposition}

\begin{proof}
Clearly, since 
\begin{equation*}
\int_{\delta _{t}c(t)}^{\infty }\underline{n}_{t}(dz)\overline{\Pi }^{\ast
}(z)\leq \overline{\Pi }^{\ast }(\delta _{t}c(t))\underline{n}(\zeta
>t)=o(t^{-1}\underline{n}(\zeta >t)),
\end{equation*}%
we have 
\begin{eqnarray}
\theta (t,y) &=&\int_{0}^{\infty }\underline{n}_{t}(dz)\overline{\Pi }^{\ast
}(z+y)  \notag \\
&=&\int_{0}^{\delta _{t}c(t)}\underline{n}_{t}(dz)\overline{\Pi }^{\ast
}(z+y)+o(t^{-1}\underline{n}(\zeta >t)).  \label{b}
\end{eqnarray}%
We can apply Proposition \ref{prop2} to get%
\begin{eqnarray*}
\int_{0}^{\delta _{t}c(t)}\underline{n}_{t}(dz)\overline{\Pi }^{\ast }(z+y)
&=&\int_{0}^{\delta _{t}c(t)}\underline{n}_{t}(dz)\int_{z+y}^{\infty }\Pi
^{\ast }(dw) \\
&=&\int_{y}^{\infty }\Pi ^{\ast }(dw)\int_{0}^{(w-y)\wedge \delta _{t}c(t)}%
\underline{n}_{t}(dz) \\
&\backsim &\frac{f(0)}{tc(t)}\int_{y}^{\infty }\Pi ^{\ast
}(dw)\int_{0}^{(w-y)\wedge \delta _{t}c(t)}U(z)dz \\
&=&\frac{f(0)}{tc(t)}\int_{0}^{\delta _{t}c(t)}U(z)\overline{\Pi }^{\ast
}(z+y)dz.
\end{eqnarray*}%
In particular, we have 
\begin{equation*}
h_{0}(t)=\theta (t,0)=\frac{f(0)L(t)}{tc(t)}+o(t^{-1}\underline{n}(\zeta
>t)),
\end{equation*}%
and since the first term $\in RV(-(1+\eta ))$ and $\eta =\overline{\rho }$
we can integrate this to give 
\begin{equation}
\frac{f(0)L(t)}{\overline{\rho }c(t)}\backsim \underline{n}^{d}(\zeta >t),
\label{y}
\end{equation}%
and hence $\theta (t,0)\backsim \overline{\rho }t^{-1}\underline{n}%
^{d}(\zeta >t).$ The result for $y>0$ then follows from (\ref{b}).
\end{proof}

\begin{remark}
The results in the following section will demonstrate that we have\linebreak 
$\underline{n}^{d}(\zeta >t)\backsim p\underline{n}(\zeta >t)$ and then (\ref%
{1.4}) follows for the case $\alpha \overline{\rho }=1.$
\end{remark}

\subsection{The continuous case}

It turns out that we need to establish some parts of Theorem \ref{thm2}
before we can conclude the proof of Theorem \ref{thm1}.

\begin{theorem}
\label{S}Suppose the drift $d^{\ast }$of $H^{\ast }$ is positive. Then
uniformly in $\Delta $ and $x>0$ such that $\frac{x}{c(t)}\rightarrow 0,$%
\begin{equation}
\mathbb{P}_{x}^{c}(T\in (t,t+\Delta ])\backsim \frac{f(0)d^{\ast }\Delta
U^{\ast }(x)}{tc(t)}\text{ as }t\rightarrow \infty ,  \label{7}
\end{equation}%
and uniformly in $\Delta $ and $x>0$
\end{theorem}

\begin{equation}
\mathbb{P}_{x}^{c}(T\in (t,t+\Delta ])=\frac{d^{\ast }\Delta \overline{n}%
(\zeta >t)}{c(t)}\text{ }({g}^{\ast }(x_{t})+o(1))\text{ as }t\rightarrow
\infty .  \label{8}
\end{equation}

\begin{proof}
We use the result, from Theorem 3.1 of \cite{gm}, which states that whenever 
$d^{\ast }>0$ the bivariate renewal function $W^{\ast }(t,x)$ is
differentiable in $x$ for each $t>0,$ and\
\begin{equation*}
\mathbb{P}_{x}^{c}(T\leq t)=d^{\ast }\frac{dW^{\ast }(t,x)}{dx}.
\end{equation*}%
Recall also from Lemma \ref{LC} that $W^{\ast
}(t,x)=a+\int_{u=0}^{t}\int_{y=0}^{x}\overline{n}_{u}(dy)du,$ so that 
\begin{equation*}
\mathbb{P}_{x}^{c}(T\in (t,t+\Delta ])=d^{\ast }\int_{t}^{t+\Delta
}\lim_{h\downarrow 0}\frac{\overline{n}_{u}((x,x+h])}{h}du.
\end{equation*}%
However, by applying Proposition \ref{prop2} to $-X$ we can approximate $%
\overline{n}_{u}((x,x+h])$ uniformly in $x$ and $h,$ and see that, given any 
$\varepsilon >0,$ for $u\in \lbrack t,t+\Delta ],$ $t$ large enough, and $%
x/c(t)$ small enough$,$ 
\begin{equation*}
\frac{(1-\varepsilon )f(0)U^{\ast }(x)}{uc(u)}\leq \lim_{h\downarrow 0}\frac{%
\overline{n}_{u}((x,x+h])}{h}\leq \frac{(1+\varepsilon )f(0)U^{\ast }(x)}{%
uc(u)}
\end{equation*}%
and then (\ref{7}) is immediate. The statement (\ref{8}) is proved in
exactly the same way, but using the approximation from Proposition \ref%
{prop1}.
\end{proof}

For the next result, we need the following identity, in which $q_{t}(z)$
(respectively $q_{t}^{\ast }(z))$ denotes the density function $\underline{n}%
_{t}^{Y}(dz)/dz$ (respectively $\overline{n}_{t}^{Y}(dz)/dz).$

\begin{lemma}\label{lem:29}
For any fixed $0<s<t,$
\begin{equation}
\int_{0}^{\infty }q_{s}(z)q_{t-s}^{\ast }(z)dz=\frac{f_{t}(0)}{t}%
=t^{-(1+\eta )}f(0).  \label{u}
\end{equation}
\end{lemma}

\begin{proof}
Specialising (\ref{lastexit0}) to the stable case and observing that, in the
stable case both the ladder time processes have zero drift gives%
\begin{equation*}
f_{t}(0)=\int_{0}^{t}du\int_{0}^{\infty }q_{u}(z)q_{t-u}^{\ast }(z)dz.
\end{equation*}%
Now we can deduce from Corollary 3 of \cite{chaumontsup2010} that $%
\int_{0}^{\infty }q_{u}(z)q_{t-u}^{\ast }(z)dz/f_{t}(0)$ is the conditional
density function of the time at which $\sup (Y_{u},0\leq u\leq t)$ occurs,
given $Y_{t}=0$. However it is well-known that the time at which the
supremum of a stable bridge occurs has a uniform distribution, see e.g. \cite{chaumont97} Th\'eor\`eme 4, and the
result (\ref{u}) follows.
\end{proof}

\begin{theorem}
\label{R}If $d^{\ast }>0$ then (\ref{1.3}) holds, viz, uniformly in $\Delta
, $%
\begin{equation}
\underline{n}^{c}(\zeta \in (t,t+\Delta ])\backsim \frac{f(0)d^{\ast }\Delta 
}{tc(t)}\text{ as }t\rightarrow \infty .  \label{c}
\end{equation}
\end{theorem}

\begin{proof}
We will actually show that $\underline{n}^{c}(\zeta \in (2t,2t+\Delta
])\backsim 2^{-(1+\eta )}f(0)d^{\ast }\Delta (tc(t))^{-1},$ which is
equivalent to the stated result. Here we use a different decomposition, viz%
\begin{eqnarray*}
\underline{n}^{c}(\zeta \in (2t,2t+\Delta ])&=&\int_{0}^{\infty }\underline{n%
}(\epsilon _{t}\in dy)\mathbb{P}_{y}^{c}(T\in (t,t+\Delta ]) \\
&=&\sum_{1}^{2}I_{r}=\sum_{1}^{2}\int_{A_{r}}\underline{n}_{t}(dy)\mathbb{P}%
_{y}(T\in (t,t+\Delta ]),
\end{eqnarray*}%
where $A_{1}=(0,D^{-1}c(t)],$ and $A_{2}=(D^{-1}c(t),\infty ).$ First we
have, using Corollary \ref{Q} and Theorem \ref{S}, 
\begin{eqnarray*}
I_{1} &=&\int_{0}^{D^{-1}c(t)}\underline{n}_{t}(dy)\mathbb{P}_{y}^{c}(T\in
(t,t+\Delta ]) \\
&\backsim &\frac{d^{\ast }{f(0)}\Delta }{tc(t)}\int_{0}^{D^{-1}c(t)}\underline{n}%
_{t}(dy)U^{\ast }(y) \\
&=&\frac{d^{\ast }{f(0)}\Delta }{tc(t)}\int_{0}^{D^{-1}c(t)}U^{\ast }(dz)%
\underline{n}(\epsilon _{t}\in (z,D^{-1}c(t)] \\
&\backsim &\frac{d^{\ast }{(f(0))^{2}}\Delta }{{(tc(t))^{2}}}\int_{0}^{D^{-1}c(t)}U^{\ast
}(dz)\int_{z}^{D^{-1}c(t)}U(y)dy.
\end{eqnarray*}%
Now, using Lemma \ref{L3} 
\begin{eqnarray*}
\int_{0}^{D^{-1}c(t)}U^{\ast }(dz)\int_{z}^{D^{-1}c(t)}U(y)dy
&=&\int_{0}^{D^{-1}c(t)}U^{\ast }(z)U(z)dz \\
\leq D^{-1}c(t)U(D^{-1}c(t))U^{\ast }(D^{-1}c(t)) &\backsim &CD^{-(1+\alpha
)}tc(t).
\end{eqnarray*}%
So we can make $\lim \sup_{t\rightarrow \infty }\Delta ^{-1}I_{1}tc(t)\leq
\varepsilon $ by choice of $D=D_{\varepsilon }$. The result will then follow
if we can show that $\lim_{D\rightarrow \infty }\lim_{t\rightarrow \infty
}tc(t)(d^{\ast }\Delta )^{-1}I_{2}=f(0).$ Using Theorem \ref{S}, Proposition %
\ref{F}, and the uniform continuity of ${g}(\cdot )$ and ${g}^{\ast }(\cdot ),$
gives 
\begin{eqnarray*}
\frac{tc(t)}{d^{\ast }\Delta }I_{2} &=&\frac{tc(t)\underline{n}(\zeta >t)}{%
d^{\ast }\Delta }\int_{D^{-1}c(t)}^{\infty }\underline{n}(\epsilon _{t}\in
dy|\zeta >t)\mathbb{P}_{y}^{c}(T\in (t,t+\Delta ]) \\
&=&t\overline{n}(\zeta >t)\underline{n}(\zeta >t)\int_{D^{-1}c(t)}^{\infty }%
\underline{n}(\epsilon _{t}\in dy|\zeta >t)({g}^{\ast }(y/c(t))+o(1)) \\
&=&t\overline{n}(\zeta >t)\underline{n}(\zeta >t)\int_{D^{-1}}^{\infty }%
\underline{n}(\epsilon _{t}\in c(t)dz|\zeta >t)({g}^{\ast }(z)+o(1)) \\
&=&\frac{1}{\Gamma (\rho )\Gamma (\overline{\rho })}\int_{D^{-1}}^{\infty
}{g}(z){g}^{\ast }(z)dz+o(1),
\end{eqnarray*}%
where we have used Lemma \ref{L3}$.$ Now since 
\begin{eqnarray*}
{g}(z)dz/\Gamma (\rho ) &=&\underline{n}^{Y}(\epsilon _{1}\in dz|\zeta >1)%
\underline{n}^{Y}(\zeta >1)=q_{1}(z)dz,\text{ and} \\
{g}^{\ast }(z)dz/\Gamma (\overline{\rho }) &=&\overline{n}^{Y}(\epsilon
_{1}\in dz|\zeta >1)\overline{n}^{Y}(\zeta >1)=q_{1}^{\ast }(z)dz,
\end{eqnarray*}%
the result follows from Lemma \ref{lem:29}.
\end{proof}

\begin{remark}
\label{EH} When $d^{\ast }>0$ and $\mathbb{E}H_{1}^{\ast }<\infty $ we see
from (\ref{c}) and (\ref{y}) that%
\begin{eqnarray*}
\underline{n}^{c}(\zeta>t)&\backsim& \frac{f(0)d^{\ast }}{\overline{\rho }%
c(t)}{\ \backsim}\ q\underline{n}(\zeta >t), \\
\text{and }\underline{n}^{d}(\zeta>t)&\backsim& \frac{f(0)A^{\ast }(\infty )%
}{\overline{\rho }c(t)}{\ \backsim}\ p\underline{n}(\zeta >t),
\end{eqnarray*}%
where to get the second estimates we used that the first estimates imply $$\underline{n}(\zeta >t)c(t)\rightarrow f(0)(d^{\ast }+A(\infty ))/\overline{\rho }.$$ Thus we can rewrite (\ref{c}) as 
\begin{equation*}
\lim_{t\rightarrow \infty }\frac{\underline{n}^{c}(t,\Delta ]}{\overline{%
\rho }\Delta \underline{n}(\zeta >t)}=\frac{d^{\ast }}{(d^{\ast }+A^{\ast
}(\infty ))}\text{ },
\end{equation*}%
and since this also holds when $A^{\ast }=\infty ,$ we recover (\ref{1.3x}). 
\end{remark}

\section{Proof of Theorem \protect\ref{thm2} and refinements}

For the case when $X$ is irregular upwards we need
\begin{lemma}
\label{drift}Assume $a^{\ast }>0.$ For $\alpha \overline{\rho }\leq 1,$ we
have that uniformly as $x/c(t)\downarrow 0,$ 
\begin{equation*}
\chi (t,x)\begin{cases}=o(U^{\ast}(x)h_{0}(t)), &\text{if}\ \alpha\overline{\rho}<1,\\\sim \frac{\overline{\rho }}{d^{\ast }+A^{\ast }(\infty )}\frac{\underline{n}(\zeta >t)}{t}{\int_{0}^{x}U^{\ast }(y)\overline{\Pi }^{\ast }(x-y)dy},&\text{if}\ \alpha\overline{\rho}=1, \end{cases}
\end{equation*}%
where the term $\overline{\rho }/(d^{\ast }+A^{\ast }(\infty ))$ is
understood as $o(1)$ when $A^{\ast }(\infty )=\infty .$ Also for any $D>0,$
uniformly in $D^{-1}c(t)<x<Dc(t),$ 
\begin{equation*}
t\chi (t,x)=o(1).
\end{equation*}
\end{lemma}

\begin{proof}
First observe that the fact that $a^{\ast}>0$ implies that $X$ is irregular upwards and, by Bertoin's test, see e.g. page 64 in \cite{doneybook}, necessarily $X$ has bounded
variation. A consequence of the bounded variation of $X$ is that 
\begin{equation*}
\int_{\mathbb{R}\setminus \{0\}}1\wedge |w|\Pi (dw)<\infty ,\quad y\overline{%
\Pi }^{\ast }(y)=o(1),\quad \text{as}\ y\rightarrow 0.
\end{equation*}
Making an integration by parts it is easily seen that 
\begin{equation*}
\chi (t,x)=\int_{0}^{\infty }\Pi ^{\ast }(dw)\overline{n}_{t}((x-w)^{+}<%
\epsilon _{t}<x).
\end{equation*}
Assume that $x_{t}\rightarrow 0$ as $t\rightarrow \infty .$ By the usual
approximation method using Lemma \ref{prop21} we have that uniformly in $%
x_{t}\rightarrow 0$ as $t\rightarrow \infty ,$ 
\begin{equation*}
\begin{split}
\chi (t,x)& \sim \frac{f(0)}{tc(t)}\left( \int_{0}^{x}\Pi ^{\ast
}(dw)\int_{(x-w)^{+}}^{x}U^{\ast }(z)dz\right) \\
& =\frac{f(0)}{tc(t)}\int_{0}^{x}U^{\ast }(z)\overline{\Pi }^{\ast }(x-z)dz.
\end{split}%
\end{equation*}%
{When $\alpha\overline{\rho}=1,$} Lemma \ref{L3} and the elementary renewal theorem imply that 
\begin{equation*}
\frac{1}{c(t)\underline{n}(\zeta >t)}\sim \frac{U^{\ast }(c(t))}{c(t)k_{1}}%
\xrightarrow[t\to\infty]{}\frac{1}{k_{1}\mathbb{E}(H_{1}^{\ast })},
\end{equation*}%
where the above is understood as zero when $\mathbb{E}(H_{1}^{\ast })=\infty
.$ Remark \ref{EH} implies that when $\alpha \overline{\rho }=1,$ then the
above limit equals $\overline{\rho }/f(0)\mathbb{E}(H_{1}^{\ast }).$ So the
result follows by equating the constants.

In the case where $\alpha\overline{\rho}<1,$ we can chose $t$ large enough such that $x<c(t)$ and thus we have that 
\begin{equation*}
\begin{split}
 \frac{t}{\underline{n}(\zeta>t)U^{\ast}(x)}\chi (t,x)&\sim\frac{f(0)}{c(t)\underline{n}(\zeta>t)}\frac{1}{U^{\ast}(x)}\int_{0}^{x}U^{\ast }(z)\overline{\Pi }^{\ast }(x-z)dz\\
&\leq\frac{f(0)\int_{0}^{x}\overline{\Pi }^{\ast }(z)dz}{c(t)\underline{n}(\zeta>t)}\\
&\sim C\frac{U^{\ast}(c(t))}{c(t)}\int_{0}^{x}\overline{\Pi }^{\ast }(z)dz\\
&\leq C\frac{\int_{0}^{c(t)}\overline{\Pi }^{\ast }(z)dz}{\int^{c(t)}_{0}\overline{\mu}^{*}(y)dy}\\
&=o(1),
\end{split}
\end{equation*}
in the third line we used Lemma \ref{L3}, in the fourth line we used Proposition~{III.1} in \cite{bertoinbook}, in the fifth line we used that $\int_{0}^{c(t)}\overline{\Pi }^{\ast }(z)dz\in RV((1-\alpha)^{+}/\alpha),$ $\int^{c(t)}_{0}\overline{\mu}^{*}(y)dy\in RV((1-\alpha\overline{\rho})/\alpha)$ and that $(1-\alpha)^{+}<(1-\alpha\overline{\rho})$.

We now deal with the case $D^{-1}c(t)<x<Dc(t).$ As before by the usual
approximation method using Lemma \ref{F} we have that 
\begin{equation*}
\begin{split}
& \chi (t,x)\sim \frac{\overline{n}(\zeta >t)}{c(t)}\int_{0}^{x}dw\overline{%
\Pi }^{\ast }(w)\left( {g^{\ast}}\left( \frac{(x-w)^{+}}{c(t)}\right)
+o(1)\right) \\
& \leq C\frac{\overline{n}(\zeta >t)}{c(t)}\int_{0}^{Dc(t)}dw\overline{\Pi }%
^{\ast }(w).
\end{split}%
\end{equation*}%
Observe that, by Karamata's Theorem, in all cases $\int_{0}^{Dc(t)}dw%
\overline{\Pi }^{\ast }(w)=o(c(t)),$ so the result follows{\color{red}.}
\end{proof}

\subsection{The small deviation case}

\begin{theorem}
\label{E}If $X$ is asymptotically stable with $\alpha \overline{\rho }<1$,
then uniformly in $x>0$ such that $x_{t}:=x/c(t)\rightarrow 0,$%
\begin{equation*}
h_{x}(t)\backsim U^{\ast }(x)h_{0}(t)\backsim p\overline{\rho }U^{\ast }(x)%
\underline{n}(\zeta >t)/t\text{ as }t\rightarrow \infty .
\end{equation*}
\end{theorem}

\begin{remark}
Since $\mathbb{E}H_{1}^{\ast }=\infty $ we know, by Theorem \ref{S} and
remark \ref{EH}, that $\underline{n}^{c}(t,\Delta]=o(\underline{n}%
^{d}(t,\Delta ])$ and $\mathbb{P}_{x}^{c}(T_{0}\in(t,t+\Delta])=o(U^{\ast }(x)\underline{n}%
(\zeta >t)/t),$  and since $p=1$ this will give the result of Theorem \ref%
{thm2} when $\alpha \overline{\rho }<1,$ and also the analogue of (\ref{1.4})%
$.$
\end{remark}

\begin{proof}
Recalling equation (\ref{x}) and Lemma \ref{drift} we can write $%
h_{x}(t)=I_{1}+I_{2}+a\theta (t,x)+{a^{\ast}}o(U^{\ast }(x)h_{0}(t))$ where%
\begin{eqnarray*}
I_{1}+a\mathbb{\theta (}t,x) &=&\int_{0}^{\delta t}ds\int_{0}^{x}\overline{n}%
_{s}(x-dy)\theta (t-s,y)+a\mathbb{\theta (}t,x) \\
&=&\int_{[0,\delta t)}\int_{[0,x]}W^{\ast }(ds,x-dy)\theta (t-s,y) \\
&\backsim &\overline{\rho }\int_{[0,\delta t)}\int_{[0,x]}W^{\ast
}(ds,x-dy)(t-s)^{-1}\underline{n}(\zeta >t-s)\phi (y/c(t-s)),
\end{eqnarray*}%
uniformly in $x,$ by Proposition \ref{Z}$.$ Since $\phi \leq 1$ {and it is a non-increasing function} we can bound
the latter {from} above by%
\begin{equation*}
\frac{\overline{\rho }\underline{n}(\zeta >t(1-\delta ))}{t(1-\delta )}%
\int_{[0,\delta t)}\int_{[0,x]}W^{\ast }(ds,x-dy)\leq \frac{\overline{\rho }%
\underline{n}(\zeta >t(1-\delta ))U^{\ast }(x)}{t(1-\delta )},
\end{equation*}%
and below by 
\begin{eqnarray*}
&&\frac{\overline{\rho }\underline{n}(\zeta >t)\phi (x/c(t))}{t}%
\int_{[0,\delta t)}\int_{[0,x]}W^{\ast }(ds,x-dy) \\
&\geq &\frac{(1-\varepsilon )\overline{\rho }\underline{n}(\zeta >t)}{t}%
\left( U^{\ast }(x)-\int_{\delta t}^{\infty }\int_{[0,x]}W^{\ast
}(ds,x-dy)\right)
\end{eqnarray*}%
for arbitrary $\varepsilon >0$ and all sufficiently large $t.$ Also, using
the result corresponding to Proposition \ref{prop2} for $-X$ 
\begin{eqnarray*}
\int_{\delta t}^{\infty }\int_{0}^{x}W^{\ast }(ds,x-dy) &=&\int_{\delta
t}^{\infty }ds\int_{0}^{x}\overline{n}_{s}(dy) \\
&\leq &C\int_{\delta t}^{\infty }ds\int_{0}^{x}U^{\ast }(y)dy/sc(s) \\
&\leq &CxU^{\ast }(x)/c({\delta}t)=o(U^{\ast }(x)),
\end{eqnarray*}%
and we conclude that 
\begin{equation*}
I_{1}+a\theta (t,x)\overset{t,\delta }{\backsim }h_{0}(t)U^{\ast }(x).
\end{equation*}%
Also, we can write $\theta (t,y)=\int_{y}^{\infty }\nu (t,dw)$ where $\nu
(t,dw)=\int_{0}^{\infty }\underline{n}_{t}(dz)\Pi ^{\ast }(dw+z).$ This
allows us to integrate $\int_{0}^{x}\overline{n}_{t-s}(dy)\theta (s,x-y)$ by
parts and apply the result for $-X$ corresponding to Corollary \ref{Q}, to
get 
\begin{eqnarray*}
I_{2} &=&\int_{0}^{(1-\delta )t}ds\int_{0}^{x}\overline{n}_{t-s}(dy)\theta
(s,x-y) \\
&\leq &\frac{C}{tc(t)}\int_{0}^{(1-\delta )t}ds\int_{0}^{x}U^{\ast
}(y)\theta (s,x-y)dy \\
&\leq &\frac{C}{tc(t)}\int_{0}^{x}U^{\ast }(y)\underline{n}\{O>x-y\}{dy} \\
&\leq &\frac{CU^{\ast }(x)A^{\ast }(x)}{tc(t)}
\end{eqnarray*}%
where we recall that $A^{\ast }(x)=\int_{0}^{x}\overline{\mu }^{\ast }(y)dy,$
$\overline{\mu }^{\ast }(y)=\underline{n}(O>y)$ is the tail of the L\'{e}vy
measure of the decreasing ladder-height process, and $U^{\ast }(x)\backsim
x/A^{\ast }(x)$ as $x\rightarrow \infty $. Since $A^{\ast }\in RV(1-\alpha 
\overline{\rho })$ we have 
\begin{eqnarray*}
A^{\ast }(x)/c(t)\underline{n}(\zeta >t)&=&o(A^{\ast }(c(t))/c(t)\underline{n%
}(\zeta >t) \\
&=&o(1/U^{\ast }(c(t))\underline{n}(\zeta >t)),
\end{eqnarray*}%
and the result follows from Lemma \ref{L3}.
\end{proof}

\begin{theorem}
\label{I}If $X$ is asymptotically stable with $\alpha \overline{\rho }=1,$
the conclusion of Theorem \ref{E} holds.
\end{theorem}

\begin{proof}
This time we write $h_{x}(t)=I_{1}+I_{2}+I_{3}+a\theta (t,x)+a^{\ast }\chi
(t,x)$ where%
\begin{eqnarray*}
I_{1}+a\theta (t,x) &=&\int_{0}^{\delta t}ds\int_{0}^{x}\overline{n}%
_{s}(x-dy)\theta (t-s,y)+a\theta (t,x) \\
&=&\int_{0}^{\delta t}\int_{(0,x]}W^{\ast }(ds,x-dy)\theta (t-s,y).
\end{eqnarray*}%
Since $\int_{0}^{\delta t}\int_{(0,x]}W^{\ast }(ds,x-dy)\leq U^{\ast }(x)$
we see from Proposition \ref{K} that, writing $\Delta _{t}=\delta _{t}c(t)$
and introducing the monotone decreasing function $\gamma (t)=\overline{\rho }%
\underline{n}(\zeta >t)/(tL(t)),$%
\begin{equation*}
I_{1}=\int_{0}^{\delta t}\int_{(0,x]}\int_{z=0}^{\Delta _{t}}W^{\ast
}(ds,x-dy)\gamma (t-s)U(z)\overline{\Pi }^{\ast }(z+y)dz+o(U^{\ast }(x)%
\underline{n}(\zeta >t)/t).
\end{equation*}%
The integral here is bounded above by $\gamma ((1-\delta )t)J(t,x)$ and
below by $\gamma (t)(J(t,x)-e(t,x))$, where%
\begin{eqnarray*}
J(t,x) &=&\int_{0<y\leq x}\int_{z=0}^{\Delta _{t}}U^{\ast }(x-dy)U(z)%
\overline{\Pi }^{\ast }(z+y)dz, \\
e(t,x) &=&\int_{\delta t}^{\infty }\int_{0<y\leq x}\int_{z=0}^{\Delta _{t}}%
\overline{n}_{s}(x-dy)U(z)\overline{\Pi }^{\ast }(z+y){dzds}.
\end{eqnarray*}%
Note that%
\begin{eqnarray*}
e(t,x) &\leq &\int_{\delta t}^{\infty }\int_{0<y\leq x}\int_{z=0}^{\Delta
_{t}}\overline{n}_{s}(x-dy)U(z)\overline{\Pi }^{\ast }(z){dzds} \\
&=&L(t)\int_{\delta t}^{\infty }\overline{n}_{s}((0,x])ds\backsim
L(t)f(0)\int_{0}^{x}U^{\ast }(y)dy\int_{\delta t}^{\infty }\frac{ds}{sc(s)}
\\
&\backsim &\frac{\alpha L(t)f(0)}{\delta ^{\eta }c(t)}\int_{0}^{x}U^{\ast
}(y)dy\leq \frac{\alpha f(0)}{\delta ^{\eta }}\frac{xU^{\ast }(x)L(t)}{c(t)}%
=o(U^{\ast }(x)\underline{n}(\zeta >t)/t),
\end{eqnarray*}%
where we have used Corollary \ref{Q} in the second line. Also 
\begin{eqnarray}
&&J(t,x)\\&=&\int_{z=0}^{\Delta _{t}}U(z)\int_{0}^{x}U^{\ast }(x-dy)\overline{%
\Pi }^{\ast }(z+y)dz  \notag \\
&=&\int_{z=0}^{\Delta _{t}}U(z)dz\left( U^{\ast }(x)\overline{\Pi }^{\ast
}(z)-\int_{0}^{x}U^{\ast }(x-y)\Pi ^{\ast }(z+dy)\right)  \notag \\
&=&U^{\ast }(x)L(t)-\int_{z=0}^{\Delta _{t}}U(z)dz\int_{z}^{z+x}U^{\ast
}(x+z-w)\Pi ^{\ast }(dw)  \notag \\
&=&U^{\ast }(x)L(t)-\int_{w=0}^{\Delta _{t}+x}\Pi ^{\ast
}(dw)\int_{(w-x)^{+}}^{w}U^{\ast }(x+z-w)U(z)dz  \notag \\
&=&U^{\ast }(x)L(t)-\int_{w=0}^{\Delta _{t}+x}\Pi ^{\ast
}(dw)\int_{(w-\Delta _{t})^{+}}^{x\wedge w}U^{\ast }(x-y)U(w-y)dy.
\label{2a}
\end{eqnarray}%
Also, using Proposition \ref{prop2} and the usual approximation argument, we
see that%
\begin{eqnarray*}
I_{3} &=&\int_{0}^{\delta t}ds\int_{0}^{x}\theta (s,y)\overline{n}%
_{t-s}(x-dy) \\
&\backsim &\int_{0}^{\delta t}\int_{0}^{x}\frac{f(0)\theta (s,y)}{(t-s)c(t-s)%
}U^{\ast }(x-y)dyds \\
&\leq &\frac{f(0)}{(1-\delta )tc((1-\delta )t)}\int_{0}^{x}\int_{0}^{\infty
}\theta (s,y)U^{\ast }(x-y)dyds.
\end{eqnarray*}%
Since 
\begin{eqnarray*}
\int_{0}^{\infty }\theta (s,y)ds &=&\int_{0}^{\infty }\int_{0}^{\infty }%
\underline{n}_{s}(dz)\overline{\Pi }^{\ast }(y+z)ds \\
&=&\int_{0}^{\infty }U(dz)\overline{\Pi }^{\ast }(y+z)-a^{\ast }\overline{%
\Pi }^{\ast }(y)=\overline{\mu }^{\ast }(y)-a^{\ast }\overline{\Pi }^{\ast
}(y),
\end{eqnarray*}%
we get that the double integral above {equals $$\int_{0}^{x}\overline{\mu }^{\ast
}(y)U^{\ast }(x-y)dy-a^{\ast }\int_{0}^{x}dy\overline{\Pi }^{\ast }(y)U^{\ast
}(x-y).$$} Noting that {$\int_{\delta t}^{\infty }\theta (s,y)ds\leq \underline{n%
}^{d}(\zeta >\delta t)$} and so
\begin{eqnarray*}
&&\frac{1}{tc(t)}\int_{0}^{x}U^{\ast }(x-y)dy\int_{\delta t}^{\infty }\theta
(s,y)ds\leq \frac{\underline{n}(\zeta >\delta t)}{tc(t)}{\int_{0}^{x}U^{\ast
}(x-y)dy} \\
&\leq &\frac{xU^{\ast }(x)\underline{n}(\zeta >\delta t)}{tc(t)}%
=o(t^{-1}U^{\ast }(x)\underline{n}(\zeta >t)),
\end{eqnarray*}%
and recalling that $f(0)/tc(t)\backsim \overline{\rho }\underline{n}%
^{d}(\zeta >t)/tL(t)=p\gamma (t)$, we see that there is a corresponding
lower bound and hence, from Lemma \ref{drift}, 
\begin{equation}
\lim_{\delta \rightarrow 0,t\rightarrow \infty }\frac{I_{3}+a^{\ast }\chi
(t,x)}{\gamma (t)K(x)}=p,\text{ where }K(x)=\int_{0}^{x}\overline{\mu }%
^{\ast }(y)U^{\ast }(x-y)dy.  \label{3a}
\end{equation}%
On the other hand, using Vigon's expression for $\overline{\mu }^{\ast }$ we
see that 
\begin{eqnarray*}
K(x) &=&\int_{0}^{x}\int_{0}^{\infty }\Pi ^{\ast }(y+dv)U(v)U^{\ast }(x-y)dy
\\
&=&\int_{0}^{\infty }\Pi ^{\ast }(du)\int_{0}^{x\wedge u}U(u-y)U^{\ast
}(x-y)dy,
\end{eqnarray*}%
and hence 
\begin{eqnarray*}
J(t,x)+K(x)-U^{\ast }(x)L(t) &=&\int_{x+\Delta _{t}}^{\infty }\Pi ^{\ast
}(du)\int_{0}^{x\wedge u}U(u-y)U^{\ast }(x-y)dy \\
&\leq &U^{\ast }(x)\int_{x+\Delta _{t}}^{\infty }\Pi ^{\ast
}(du)\int_{0}^{(x+\Delta _{t})}U(u-y)dy \\
&=&U^{\ast }(x)E(x+\Delta _{t})=o(U^{\ast }(x)A^{\ast }(x+\Delta _{t})),
\end{eqnarray*}%
by Lemma \ref{L}. But for large $t$ we have $\Delta _{t}\leq x+\Delta
_{t}\leq c(t),$ so $A^{\ast }(x+\Delta _{t})\backsim L(t).$ Then it follows
from (\ref{2a}) and (\ref{3a}) that, uniformly in $x,$ 
\begin{equation*}
\lim_{\delta \rightarrow 0,t\rightarrow \infty }\frac{t(I_{1}+I_{3}{+a^{\ast}\chi(t,x)})}{%
\overline{\rho }U^{\ast }(x)\underline{n}(\zeta >t)}=p.
\end{equation*}%
It is also straight forward to check that, for any fixed $\delta \in
(0,1/2), $ $I_{2}=o(t^{-1}U^{\ast }(x)\underline{n}(\zeta >t)),$ and the
result follows.
\end{proof}

\subsection{Normal deviations}

Again we start with a preparatory result.

\begin{lemma}
\label{o}If $\alpha \overline{\rho }<1,$ the identity 
\begin{equation}
\widetilde{h}_{x}(1)=\frac{\overline{\rho }}{\Gamma (\overline{\rho })\Gamma
(\rho )}\int_{0}^{1}ds\int_{0}^{x}dy\phi \left( (x-y)(1-s)^{-\eta }\right)
)(1-s)^{-\overline{\rho }-1}{g}^{\ast }\left( ys^{-\eta }\right) s^{-\rho
-\eta },  \label{h}
\end{equation}%
holds for $x>0$, where $\phi $ is defined in Proposition \ref{Z} and $\tilde{%
h}_{x}$ is the downwards first passage density for $Y$ starting from $x>0.$
\end{lemma}

\begin{proof}
Recall that $\phi (z)=\mathbb{E}(z+Z_{1})^{-\alpha }/\mathbb{E}%
Z_{1}^{-\alpha }=k^{\ast }\mathbb{E}(z+Z_{1})^{-\alpha }/\overline{\rho },$
where we have used (\ref{22}). Also the left-hand tail of the L\'{e}vy
measure of $Y$ is $k^{\ast }x^{-\alpha },$ so if we write the equation (\ref%
{x}) for $Y$ with $t=1$ we have%
\begin{eqnarray*}
\widetilde{h}_{x}(1) &=&k^{\ast }\int_{0}^{1}ds\int_{0}^{x}\overline{n}%
_{s}^{Y}(dy)\underline{n}^{Y}((x-y+\epsilon _{1-s})^{-\alpha },\zeta >1-s) \\
&=&k^{\ast }\int_{0}^{1}ds\int_{0}^{x}\int_{0}^{\infty }q_{s}^{\ast
}(y)(x-y+z)^{-\alpha }q_{1-s}(z)dydz.
\end{eqnarray*}%
Using (\ref{23}) and its analogue for $q^{\ast },$ and recalling that $%
\underline{n}^{Y}(\zeta >1)\overline{n}^{Y}(\zeta >1)=(\Gamma (\overline{%
\rho })\Gamma (\rho ))^{-1}$ the RHS becomes%
\begin{eqnarray*}
&&\frac{k^{\ast }}{\Gamma (\overline{\rho })\Gamma (\rho )}%
\int_{0}^{1}ds\int_{0}^{x}\int_{0}^{\infty }(x-y+z)^{-\alpha }s^{-\eta -\rho
}{g}^{\ast }(ys^{-\eta })(1-s)^{-\eta -\overline{\rho }}{g}(z(1-s)^{-\eta })dydz
\\
&=&\frac{k^{\ast }}{\Gamma (\overline{\rho })\Gamma (\rho )}%
\int_{0}^{1}ds\int_{0}^{x}\int_{0}^{\infty }(x-y+w(1-s)^{\eta })^{-\alpha
}s^{-\eta -\rho }{g}^{\ast }(ys^{-\eta })(1-s)^{-\overline{\rho }}{g}(w)dydw \\
&=&\frac{k^{\ast }}{\Gamma (\overline{\rho })\Gamma (\rho )}%
\int_{0}^{1}ds\int_{0}^{x}\int_{0}^{\infty }((x-y)(1-s)^{-\eta }+w)^{-\alpha
}s^{-\eta -\rho }{g}^{\ast }(ys^{-\eta })(1-s)^{-1-\overline{\rho }}{g}(w)dydw \\
&=&\frac{\overline{\rho }}{\Gamma (\overline{\rho })\Gamma (\rho )}%
\int_{0}^{1}ds\int_{0}^{x}\phi ((x-y)(1-s)^{-\eta })s^{-\eta -\rho }{g}^{\ast
}(ys^{-\eta })(1-s)^{-1-\overline{\rho }}dy,
\end{eqnarray*}%
and the result follows.
\end{proof}

\begin{theorem}
\label{G}Assume $\alpha \overline{\rho }<1.$ Then uniformly for $x_{t}\in
\lbrack D^{-1},D],$%
\begin{equation*}
th_{x}(t)=p\tilde{h}_{x_{t}}(1)+o(1)\text{ as }t\rightarrow \infty .
\end{equation*}
\end{theorem}

\begin{proof}
Recall again that $p=1$ in this situation. We use the same decomposition as
in the proof of Theorem \ref{E}. Then%
\begin{eqnarray*}
I_{1}+a\theta (t,x) &=&\int_{0}^{\delta t}ds\int_{0}^{x}\overline{n}%
_{s}(x-dy)\theta (t-s,y)+a\theta (t,x) \\
&=&\int_{0}^{\delta t}\int_{(0,x]}W^{\ast }(ds,x-dy)\theta (t-s,y) \\
&\leq &\int_{0}^{\delta t}W^{\ast }(ds,[0,\infty ))\theta (t-s,0) \\
&\leq &Ch_{0}((1-\delta )t)V^{\ast }(\delta t)\backsim C\delta ^{\overline{%
\rho }}t^{-1}\underline{n}(\zeta >t)V^{\ast }(t) \\
&\backsim &C\delta ^{\overline{\rho }}t^{-1}.
\end{eqnarray*}%
(Recall that $V^{\ast }$ is the potential function in the decreasing ladder
time process.) Next, take $0<\gamma <D^{-1},$and write $%
I_{3}=I_{3}^{1}+I_{3}^{2},$ where%
\begin{eqnarray*}
I_{3}^{1} &=&\int_{(1-\delta )t}^{t}ds\int_{0}^{\gamma c(t)}\overline{n}%
_{s}(x-dy)\theta (t-s,y) \\
&=&\int_{(1-\delta )t}^{t}{ds}\int_{0}^{\gamma
c(t)}{\overline{n}_{s}(x-dy)}\int_{0}^{\infty }\underline{n}_{t-s}(du)\overline{\Pi }^{\ast }(y+u) \\
&=&\int_{(1-\delta )t}^{t}{ds}\int_{0}^{\infty }\underline{n}_{t-s}(du)%
\int_{0}^{\gamma c(t)}\overline{n}_{s}(x-dy)\int_{y+u}^{\infty }\Pi ^{\ast
}(dw) \\
&=&\int_{(1-\delta )t}^{t}{ds}\int_{0}^{\infty }\underline{n}_{t-s}(du)%
\int_{u}^{\infty }\Pi ^{\ast }(dw)\int_{0}^{\gamma c(t)\wedge (w-u)}%
\overline{n}_{s}(x-dy).
\end{eqnarray*}%
From Corollary \ref{xx} we see that for all $\gamma >0$ and all $s\geq
(1-\delta )t$ and all sufficiently large $t,$%
\begin{equation*}
\int_{0}^{\gamma c(t)\wedge (w-u)}\overline{n}_{s}(x-dy)\leq \frac{C%
\overline{n}(\zeta >t)\int_{0}^{\gamma c(t)\wedge (w-u)}dy}{c(t)},
\end{equation*}%
and hence%
\begin{eqnarray*}
\int_{u}^{\infty }\Pi ^{\ast }(dw)\int_{0}^{\gamma c(t)\wedge (w-u)}%
\overline{n}_{s}(x-dy) &\leq &\frac{C\overline{n}(\zeta >t)\int_{u}^{\infty
}\Pi ^{\ast }(dw)\int_{0}^{\gamma c(t)\wedge (w-u)}dy}{c(t)} \\
&=&\frac{C\overline{n}(\zeta >t)\int_{0}^{\gamma c(t)}dy\overline{\Pi }%
^{\ast }(u+y)}{c(t)}.
\end{eqnarray*}%
Thus%
\begin{eqnarray*}
c(t)I_{3}^{1} &\leq &C\overline{n}(\zeta >t)\int_{0}^{\delta
t}ds\int_{0}^{\infty }\underline{n}_{s}(du)\int_{0}^{\gamma c(t)}\overline{%
\Pi }^{\ast }(u+y)dy \\
&=&C\overline{n}(\zeta >t)\int_{0}^{\delta t}\int_{0}^{\infty
}W(ds,du)\int_{0}^{\gamma c(t)}\overline{\Pi }^{\ast }(u+y)dy \\
&\leq &C\overline{n}(\zeta >t)\int_{0}^{\infty }U(du)\int_{0}^{\gamma c(t)}%
\overline{\Pi }^{\ast }(u+y)dy \\
&=&C\overline{n}(\zeta >t)\int_{z=0}^{\gamma c(t)}\overline{\mu }^{\ast
}(z)dz\backsim C\overline{n}(\zeta >t)\gamma c(t)\overline{\mu }^{\ast
}(\gamma c(t)) \\
&\backsim &\frac{C\gamma \overline{n}(\zeta >t)c(t)}{U^{\ast }(\gamma c(t))}%
\backsim \frac{C\gamma ^{1-\alpha \overline{\rho }}\overline{n}(\zeta >t)c(t)%
}{U^{\ast }(c(t))} \\
&\backsim &C\gamma ^{1-\alpha \overline{\rho }}\overline{n}(\zeta >t)c(t)%
\underline{n}(\zeta >t)\backsim C\gamma ^{1-\alpha \overline{\rho }}c(t){t^{-1}}.
\end{eqnarray*}%
Thus $\lim_{\gamma \rightarrow 0}\lim \sup tI_{3}^{1}=0.$ Also%
\begin{eqnarray*}
I_{3}^{2} &=&\int_{(1-\delta )t}^{t}ds\int_{\gamma c(t)}^{x}\overline{n}%
_{s}(x-dy)\underline{n}(\overline{\Pi }^{\ast }(y+\epsilon _{t-s}),t-s<\zeta
) \\
&\leq &\overline{\Pi }^{\ast }(\gamma c(t))\int_{(1-\delta
)t}^{t}ds\int_{\gamma c(t)}^{x}\overline{n}_{s}(x-dy)\underline{n}(\zeta
>t-s) \\
&\leq &\overline{\Pi }^{\ast }(\gamma c(t)){\mathbb{P}}(G_{t}\geq (1-\delta )t),
\end{eqnarray*}%
where $G_{t}$, the time of the last zero of $X-I$ before $t,$ has the
property that $t^{-1}G_{t}$ has a limiting arc-sine distribution of index $%
\overline{\rho }$. (See Theorem 14, p 169 of \cite{bertoinbook}.) It follows
that for each fixed $\gamma >0,$ we have $\lim_{\delta \rightarrow 0}\lim
\sup_{t\rightarrow \infty }tI_{3}^{2}=0,$ and hence $\lim_{\delta
\rightarrow 0}\lim \sup_{t\rightarrow \infty }t(I_{1}+I_{3})=0,$ uniformly
in $x.$ The term $a^{\ast }\chi (t,x)$ is $o(t^{-1})$ by Lemma \ref{drift}.
Using the bounds
\begin{eqnarray*}
tI_{2} &\geq& t\int_{\delta t}^{(1-\delta )t}ds\sum_{0}^{[x]}\overline{n}%
_{s}((r,r+1])\theta (t-s,(x-r)) \\
tI_{2}&\leq& t\int_{\delta t}^{(1-\delta )t}ds\sum_{0}^{[x]}\overline{n}%
_{s}((r,r+1])\theta (t-s,(x-r-1)^{+})
\end{eqnarray*}%
and Propositions \ref{Z} and \ref{F}, for any $\delta >0,$ we can estimate $%
tI_{2}$ by 
\begin{equation*}
k_{3}\overline{\rho }\int_{\delta t}^{(1-\delta )t}ds\sum_{0}^{[x]}\frac{%
{g}^{\ast }(r/c(s)\overline{n}(\zeta >s)\underline{n}(\zeta >t-s)\phi
((x-r-1)^{+}/c\left( t-s\right) )}{c(s)\overline{n}(\zeta >t)\underline{n}%
(\zeta >t)(t-s)}(1+o(1)),
\end{equation*}%
where the error term is uniform in $x.$ Putting $r=c(t)z$ and $s=tu$ we get
the uniform estimate
\begin{eqnarray*}
&&k_{3}\overline{\rho }\int_{\delta }^{1-\delta }\int_{0}^{x_{t}}{g}^{\ast
}(zu^{-\eta })u^{-(\eta +\rho )}\phi ((x_{t}-z)(1-u)^{-\eta })(1-u)^{-1-%
\overline{\rho }}dudz+o(1) \\
&&:=I(\delta ,x_{t})+o(1).
\end{eqnarray*}%
Next, we show that, as $\delta \rightarrow 0,$ $I(\delta ,w)=I(0,w)+o(1),$
uniformly in $w.$ First, since $\phi $ is bounded, for small $\delta $%
\begin{eqnarray*}
&&\int_{0}^{\delta }\int_{0}^{w}{g}^{\ast }(zu^{-\eta })u^{-(\eta +\rho )}\phi
((w-z)(1-u)^{-\eta })(1-u)^{-(2-\rho )}dudz \\
&\leq &C\int_{0}^{\delta }\int_{0}^{w}{g}^{\ast }(zu^{-\eta })u^{-(\eta +\rho
)}dudz=c\int_{0}^{\delta }\int_{0}^{wu^{-\eta }}{g}^{\ast }(y)u^{-\rho }dudy \\
&\leq &C\int_{0}^{\delta }\int_{0}^{\infty }{g}^{\ast }(y)u^{-\rho
}dudy\rightarrow 0\text{ as }\delta \rightarrow 0.
\end{eqnarray*}%
Also ${g}^{\ast }$ is bounded, so the same argument shows that the
contribution from $(1-\delta ,1)$ is bounded above by $C\int_{0}^{\delta
}\int_{0}^{Du^{-\eta }}\phi (z)u^{\eta +\rho -2}dudz.$ By considering
separately the cases $\alpha <1,\alpha =1,$ and $\alpha >1,$ it is easy to
check that this is also finite and $\rightarrow 0$ as $\delta \rightarrow 0,$
and then the result follows from Lemma \ref{o}.
\end{proof}

\begin{theorem}
\label{H}If $X$ is asymptotically stable with $\alpha \overline{\rho }=1,$
then uniformly for $x_{t}\in \lbrack D^{-1},D],$%
\begin{equation*}
h_{x}(t)=\frac{\overline{n}(\zeta >t)L(t)}{c(t)}({g}^{\ast }(x_{t})+o(1))\text{
as }t\rightarrow \infty .
\end{equation*}
\end{theorem}

\begin{proof}
Notice that, by (\ref{y}) and Remark \ref{EH}%
\begin{equation}\label{eq:53}
\frac{t\overline{n}(\zeta >t)L(t)}{c(t)}\backsim \overline{\rho }f(0)t%
\overline{n}(\zeta >t)\underline{n}^{d}(\zeta >t)\rightarrow pk_{3}\overline{%
\rho }f(0):=k_{6},
\end{equation}%
so we will prove that $th_{x}(t)=k_{6}{g}^{\ast }(x_{t})+o(1).$ This time we
write%
\begin{eqnarray*}
h_{x}(t) &=&\int_{0}^{t}\int_{0}^{x}\overline{n}_{s}(x-dy)\theta
(t-s,y)+a\theta (t,x)+a^{\ast }\chi (t,x) \\
&=&\sum_{1}^{4}J_{r}+a\theta (t,x)+a^{\ast }\chi (t,x),
\end{eqnarray*}%
where 
\begin{eqnarray*}
J_{1} &=&\int_{0}^{\delta t}\int_{\Delta _{t}}^{x}\overline{n}%
_{s}(x-dy)\theta (t-s,y) \\
&\leq &\int_{0}^{\delta t}\int_{\Delta _{t}}^{x}\overline{n}_{s}(x-dy)\theta
(t-s,0) \\
&\leq &\frac{C\underline{n}(\zeta >(1-\delta t))}{(1-\delta )t}%
\int_{0}^{\delta t}\int_{\Delta _{t}}^{x}W^{\ast }(ds,x-dy) \\
&\leq &\frac{C\underline{n}(\zeta >(1-\delta t))U^{\ast }(\Delta _{t})}{%
(1-\delta )t},
\end{eqnarray*} where $\Delta_{t}=c(t)\delta_{t}$ and $\delta_{t}$ has been defined before Proposition \ref{K}. 
Since $U^{\ast }\in RV(1)$ we see that $U^{\ast }(\Delta _{t})=o(U^{\ast
}(c(t))=o((\underline{n}(\zeta >t))^{-1}),$ so $\lim_{t\rightarrow \infty
}tJ_{1}=0$ for any fixed $\delta >0.$ Next, we can use Proposition \ref%
{prop1} and the usual approximation procedure to see that 
\begin{eqnarray*}
tJ_{2} &=&\int_{\delta t}^{t}\int_{\Delta _{t}}^{x}\overline{n}%
_{s}(x-dy)\theta (t-s,y)ds \\
&\backsim &t\int_{0}^{(1-\delta )t}\int_{\Delta _{t}}^{x}\frac{\overline{n}%
(\zeta >t-s){g}^{\ast }((x-y)/c(t-s))\theta (s,y)}{c(t-s)}dyds \\
&\leq &\frac{Ct\overline{n}(\zeta >(1-\delta )t)}{\delta c(\delta t)}%
\int_{0}^{\infty }\int_{\Delta _{t}}^{x}\theta (s,y)dyds \\
&\leq &\frac{Ct\overline{n}(\zeta >(1-\delta )t)}{c(\delta t)}\int_{\Delta
_{t}}^{Dc(t)}\overline{n}(O>y)dy \\
&\backsim &\frac{C(A^{\ast }(Dc(t))-A^{\ast }(\Delta _{t}))}{c(t)\underline{n%
}(\zeta >t)} \\
&\backsim &{\frac{CU^{\ast}(c(t))(A^{\ast }(Dc(t))-A^{\ast }(\Delta _{t}))}{\delta ^{\rho}(1-\delta )^{\eta }c(t)}}\\
&\backsim &\frac{C(A^{\ast }(Dc(t))-A^{\ast }(\Delta _{t}))}{\delta ^{\rho
}(1-\delta )^{\eta }{A^{\ast }(c(t))}}\rightarrow 0,
\end{eqnarray*} %
again for any fixed $\delta >0.$ (In the final step we have used (\ref{y})
and Lemma \ref{L}.) Also%
\begin{eqnarray*}
tJ_{3} &=&t\int_{0}^{(1-\delta )t}\int_{0}^{\Delta _{t}}\overline{n}%
_{s}(x-dy)\theta (t-s,y)ds \\
&\leq &t\int_{0}^{(1-\delta )t}\int_{0}^{\Delta _{t}}W^{\ast
}(ds,x-dy)\theta (t-s,0) \\
&\leq &Cth_{0}(\delta t)\int_{0}^{\infty }\int_{0}^{\Delta _{t}}W^{\ast
}(ds,x-dy) \\
&\backsim &C\delta ^{-(1+\overline{\rho })}\underline{n}(\zeta >t)(U^{\ast
}(x)-U^{\ast }(x-\Delta _{t}))\\
&\leq& C\delta ^{-(1+\overline{\rho })}%
\underline{n}(\zeta >t)U^{\ast }(\Delta _{t}) \\
&\backsim &C\delta ^{-(1+\overline{\rho })}\frac{C\delta ^{-(1+\overline{%
\rho })}U^{\ast }(\delta _{t}c(t))}{U^{\ast }(c(t))}\backsim C\delta ^{-(1+%
\overline{\rho })}\delta _{t}\rightarrow 0.
\end{eqnarray*}
Finally, arguing as for $J_{2}$ gives%
\begin{eqnarray*}
tJ_{4} &=&t\int_{(1-\delta )t}^{t}\int_{0}^{\Delta _{t}}\overline{n}%
_{s}(x-dy)\theta (t-s,y)ds \\
&\backsim &t\int_{0}^{\delta t}\int_{0}^{\Delta _{t}}\frac{\overline{n}(\zeta
>t-s){g}^{\ast }((x-y)/c(t-s))\theta (s,y)}{c(t-s)}dyds \\
&\backsim &t{g}^{\ast }(x_{t})\int_{0}^{\delta t}\int_{0}^{\Delta _{t}}\frac{%
\overline{n}(\zeta >t-s)\theta (s,y)}{c(t-s)}dyds.
\end{eqnarray*}%
An upper bound for the integral here is 
\begin{eqnarray*}
&&\frac{\overline{n}(\zeta >(1-\delta )t)}{c((1-\delta )t)}\int_{0}^{\infty
}\int_{0}^{\Delta _{t}}\theta (s,y)dyds \\
&=&\frac{\overline{n}(\zeta >(1-\delta )t)}{c((1-\delta )t)}\int_{0}^{\Delta
_{t}}\overline{\mu }^{\ast }(y)dy\backsim \frac{\overline{n}(\zeta >t)L(t)}{%
(1-\delta )c(t)}.
\end{eqnarray*}%
An asymptotic lower bound is%
\begin{equation*}
\frac{\overline{n}(\zeta >t)}{c(t)}\left( \int_{0}^{\Delta _{t}}\overline{\mu }%
^{\ast }(y)dy-\int_{\delta t}^{\infty }\int_{0}^{\Delta _{t}}\theta
(s,y)dyds\right) ,
\end{equation*}%
and since 
\begin{equation*}
\begin{split}
\int_{\delta t}^{\infty }\int_{0}^{\Delta _{t}}\theta (s,y)dyds &\leq
\int_{\delta t}^{\infty }\int_{0}^{\Delta _{t}}\theta (s,0)dyds \\
&=\Delta _{t}\int_{\delta t}^{\infty }h_{0}(s)ds\\&= \delta _{t}c(t)\underline{n}^{d}(\zeta >\delta t)\\&\backsim (\overline{\rho }f(0))^{-1}\delta ^{-\overline{\rho }}\delta _{t}L(t),
\end{split}
\end{equation*}%
it follows that $\lim_{\delta \rightarrow 0,t\rightarrow \infty }\frac{tJ_{4}%
}{{g}^{\ast }(x_{t})}=k_{6},$ uniformly for $x_{t}\in \lbrack D^{-1},D].$ The
result follows, using Lemma \ref{drift} to estimate $\chi(t,x)$
\end{proof}

\begin{corollary}
Whenever $\Pi ((-\infty ,0))>0$ we have%
\begin{equation}
th_{x}(t)=p\tilde{h}_{x_{t}}(1)+o(1)\text{ as }t\rightarrow \infty ,
\label{z}
\end{equation}%
and in all cases (\ref{1.6}) of Theorem \ref{thm2} holds.
\end{corollary}

\begin{proof}
We have proved (\ref{z}) for the case $\alpha \overline{\rho }<1$ in Theorem %
\ref{G}, and in Proposition 14 of \cite{rad} it was shown that when $\alpha 
\overline{\rho }=1$ there is a constant $k_{7}$ such that ${g}^{\ast }(x)=k_{7}%
\tilde{h}_{x}(1),$ so in this case we need to check that $k_{6}k_{7}=p.$ But
we have, from Theorems \ref{S} and \ref{H},%
\begin{eqnarray*}
t\mathbb{P}_{x}^{c}(T &\in &(t,t+\Delta ])\backsim \frac{d^{\ast }\Delta
k_{6}k_{7}\tilde{h}_{x}(1)}{L(t)}, \\
t\mathbb{P}_{x}^{d}(T &\in &(t,t+\Delta ])\backsim \Delta k_{6}k_{7}\tilde{h}%
_{x}(1).
\end{eqnarray*}%
If $p=1,$ i.e. $d^{\ast }=0$ or $d^{\ast }>0$ and $L(\infty )=\infty ,$ this
gives $t\mathbb{P}_{x}(T\in (t,t+\Delta ])\backsim \Delta k_{6}k_{7}\tilde{h}%
_{x}(1),$ and this is easily seen to contradict the standard stable
functional limit theorem unless $k_{6}k_{7}=1.$ If $p=d^{\ast }/(d^{\ast
}+L(\infty ))<1$ we get $t\mathbb{P}_{x}(T\in (t,t+\Delta ])\backsim
p^{-1}\Delta k_{6}k_{7}\tilde{h}_{x}(1)$ and the same argument gives $%
k_{6}k_{7}=p,$ and the results follow.
\end{proof}

\end{document}